\documentclass[11pt,reqno,tbtags,a4paper]{amsart}
\usepackage{amssymb}
\usepackage{url}
\usepackage[square,numbers]{natbib}
\bibpunct[, ]{[}{]}{;}{n}{,}{,}

\title[A.s.\ convergence for P\'olya urns] 
{A.s.\ convergence 
for infinite colour P\'olya urns associated with random walks}

\date{12 March, 2018}

\author{Svante Janson}
\thanks{Partly supported by the Knut and Alice Wallenberg Foundation}
\address{Department of Mathematics, Uppsala University, PO Box 480,
SE-751~06 Uppsala, Sweden}
\email{svante.janson@math.uu.se}
\urladdr{http://www.math.uu.se/svante-janson}


\numberwithin{equation}{section}

\renewcommand\le{\leqslant}
\renewcommand\ge{\geqslant}

\allowdisplaybreaks

\theoremstyle{plain}
\newtheorem{theorem}{Theorem}[section]
\newtheorem{lemma}[theorem]{Lemma}

\theoremstyle{definition}
\newtheorem{example}[theorem]{Example}

\newtheorem{problem}[theorem]{Problem}
\newtheorem{remark}[theorem]{Remark}

\theoremstyle{remark}

\newenvironment{romenumerate}[1][-10pt]{
\addtolength{\leftmargini}{#1}\begin{enumerate}
 \renewcommand{\labelenumi}{\textup{(\roman{enumi})}}%
 }{\end{enumerate}}

\newenvironment{PXenumerate}[1]{
\addtolength{\leftmargini}{1pt}%
\begin{enumerate}
 \renewcommand{\labelenumi}{\textup{(#1\arabic{enumi})}}%
 }{\end{enumerate}}

\newcounter{oldenumi}
{\setcounter{oldenumi}{\value{enumi}}
\begin{romenumerate} \setcounter{enumi}{\value{oldenumi}}}
{\end{romenumerate}}

\newcounter{thmenumerate}
\newenvironment{thmenumerate}
{\setcounter{thmenumerate}{0}%
 \def\item{\par
 \refstepcounter{thmenumerate}\textup{(\roman{thmenumerate})\enspace}}
}
{}

\newcounter{xenumerate}   

\newcommand\pfitemx[1]{\par#1:}
\newcommand\pfitemref[1]{\pfitemx{\ref{#1}}}

\newcounter{steps}
\newcommand\stepx{\smallskip\noindent\refstepcounter{steps}%
 \emph{Step \arabic{steps}: }\noindent}

\newcommand{\refT}[1]{Theorem~\ref{#1}}
\newcommand{\refTs}[1]{Theorems~\ref{#1}}

\newcommand{\refL}[1]{Lemma~\ref{#1}}
\newcommand{\refR}[1]{Remark~\ref{#1}}
\newcommand{\refS}[1]{Section~\ref{#1}}
\newcommand{\refSs}[1]{Sections~\ref{#1}}
\newcommand{\refSS}[1]{Section~\ref{#1}}

\newcommand{\refE}[1]{Example~\ref{#1}}

\newcommand\REM[1]{{\raggedright\texttt{[#1]}\par\marginal{XXX}}}
\newcommand\XREM[1]{\relax}

\begingroup
  \divide\count255 by 60
  \multiply\count255 by -60
  \advance\count255 by \time
  \ifnum \count255 < 10 \xdef\klockan{\the\count1.0\the\count255}
  \else\xdef\klockan{\the\count1.\the\count255}\fi
\endgroup

\newcommand{\sumk}{\sum_{k=1}^\infty}

\newcommand{\sumkn}{\sum_{k=1}^n}

\newcommand\set[1]{\ensuremath{\{#1\}}}
\newcommand\bigset[1]{\ensuremath{\bigl\{#1\bigr\}}}

\newcommand\xpar[1]{(#1)}
\newcommand\bigpar[1]{\bigl(#1\bigr)}
\newcommand\Bigpar[1]{\Bigl(#1\Bigr)}

\newcommand\bigsqpar[1]{\bigl[#1\bigr]}

\newcommand\bigabs[1]{\bigl|#1\bigr|}
\newcommand\Bigabs[1]{\Bigl|#1\Bigr|}
\newcommand\biggabs[1]{\biggl|#1\biggr|}
\newcommand\lrabs[1]{\left|#1\right|}
\def\rompar(#1){\textup(#1\textup)}    
\newcommand\xfrac[2]{#1/#2}

\newcommand\Bigparfrac[2]{\Bigpar{\frac{#1}{#2}}}

\def\xexp(#1){e^{#1}}

\newcommand\ntoo{\ensuremath{{n\to\infty}}}

\newcommand\ttoo{\ensuremath{{t\to\infty}}}

\newcommand\norm[1]{\|#1\|}

\newcommand\punkt{.\spacefactor=1000}    
\newcommand\iid{i.i.d\punkt}    
\newcommand\ie{i.e\punkt}
\newcommand\eg{e.g\punkt}

\newcommand\cf{cf\punkt}
\newcommand{\as}{a.s\punkt}
\newcommand{\aex}{a.e\punkt}

\newcommand\ii{\mathrm{i}}

\newcommand{\tend}{\longrightarrow}
\newcommand\dto{\overset{\mathrm{d}}{\tend}}
\newcommand\pto{\overset{\mathrm{p}}{\tend}}
\newcommand\asto{\overset{\mathrm{a.s.}}{\tend}}
\newcommand\eqd{\overset{\mathrm{d}}{=}}

\newcommand\bbR{\mathbb R}
\newcommand\bbC{\mathbb C}
\newcommand\bbN{\mathbb N}

\newcounter{MM}
\newcounter{CC}
\newcounter{cc}

\renewcommand\Re{\operatorname{Re}}

\newcommand\E{\operatorname{\mathbb E{}}}
\renewcommand\P{\operatorname{\mathbb P{}}}

\newcommand\Var{\operatorname{Var}}

\newcommand\Exp{\operatorname{Exp}}
\newcommand\Po{\operatorname{Po}}

\newcommand\ga{\alpha}

\newcommand\gd{\delta}
\newcommand\gD{\Delta}
\newcommand\gf{\varphi}

\newcommand\gG{\Gamma}

\newcommand\gL{\Lambda}
\newcommand\go{\omega}

\newcommand\gs{\sigma}

\newcommand\gss{\sigma^2}

\newcommand\eps{\varepsilon}

\renewcommand\phi{\xxx}  

\newcommand\cF{\mathcal F}

\newcommand\cL{{\mathcal L}}
\newcommand\cM{\mathcal M}

\newcommand\cP{\mathcal P}

\newcommand\cR{{\mathcal R}}
\newcommand\cS{{\mathcal S}}
\newcommand\cT{{\mathcal T}}

\newcommand\qw{^{-1}}

\newcommand\qq{^{1/2}}

\newcommand\oi{\ensuremath{[0,1]}}

\newcommand\dd{\,\mathrm{d}}

\newcommand{\chf}{characteristic function}

\newcommand\lhs{left-hand side}
\newcommand\rhs{right-hand side}

\newcommand\bmu{\tmu}
\newcommand\hbmu{\widehat{\bmu}}
\newcommand\hbmuit{\widehat{\bmui_t}}
\newcommand\hbmuet{\widehat{\bmue_t}}
\newcommand\Thetaab{\Theta_{a,b}}
\newcommand\Thetaanbn{\Theta_{a_n,b_n}}
\newcommand\ddt{\frac{\partial}{\partial t}}

\newcommand\dds{\frac{\partial}{\partial s}}

\newcommand\gsf{$\gs$-field}
\newcommand\RNP{Radon--Nikod\'ym property}
\newcommand\WW{W_1^2}
\newcommand\WWk{W_k^2}
\newcommand\WWJ{W_1^2(J)}
\newcommand\normWWJ[1]{\norm{#1}_{\WWJ}}
\newcommand\normLLJ[1]{\norm{#1}_{L^2(J)}}
\newcommand\JO{J^\circ}
\newcommand\rrt{random recursive tree}
\newcommand\wrrt{weighted random recursive tree}
\newcommand\bst{binary search tree}
\newcommand\vecs{s}
\newcommand\ctt{\cT_t}
\newcommand\bm{{m}}
\newcommand\tr{^{\mathsf{t}}}
\newcommand\ti{T^\mathsf{i}}
\newcommand\ww{\omega}
\newcommand\subL{_{\mathsf L}}
\newcommand\subR{_{\mathsf R}}
\newcommand\vl{{v\subL}}
\newcommand\vr{{v_{\mathsf R}}}
\newcommand\gfl{\gf\subL}
\newcommand\gfr{\gf\subR}

\newcommand\etal{\eta\subL}
\newcommand\etar{\eta\subR}
\newcommand\etax{\eta^*}
\newcommand\gfx{\gf^*}
\newcommand\tgf{\tilde\gf}
\newcommand\nux{\nu^*}
\newcommand\supi{^{\mathsf i}}
\newcommand\supe{^{\mathsf e}}
\newcommand\Ti{T\supi}
\newcommand\bmui{\bmu\supi}
\newcommand\Te{T\supe}
\newcommand\bmue{\bmu\supe}
\newcommand\Fi{F\supi}
\newcommand\ctti{\ctt\supi}
\newcommand\Fe{F\supe}
\newcommand\ctte{\ctt\supe}
\newcommand\tqq{\sqrt{t}}
\newcommand\gsx{\bar\gs}
\newcommand\cMx{\cM_*}
\newcommand\tmu{\widetilde \mu}
\newcommand\hmu{\widehat \mu}
\newcommand\mux{\overline{\mu}}
\newcommand\tmux{\widetilde\mux}
\newcommand\rr{\mathbf{r}}
\newcommand\hX{\widehat X}
\newcommand\ZZ{\gD Z}
\newcommand\san{\Bigparfrac{s}{a_n}}
\newcommand\xsan{\xpar{\xfrac{s}{a_n}}}
\newcommand\SRW{SBARW}
\newcommand\DRW{DARW}
\newcommand\PU{\Polya{} urn}
\newcommand\SRWPU{\SRW{} \PU}
\newcommand\DRWPU{\DRW{} \PU}

\newcommand{\Polya}{P\'olya}
\newcommand\CS{Cauchy--Schwarz}
\newcommand\CSineq{\CS{} inequality}

\begin{document}

\begin{abstract} 
 We consider P\'olya urns with infinitely many colours that are of a random
 walk type,
in two related version. We show that the colour distribution a.s., after
rescaling, converges to a normal distribution, assuming only second moments
on the offset distribution. This improves results by
Bandyopadhyay and Thacker (2014--2017; convergence in probability), and
Mailler and Marckert (2017; a.s.\ convergence assuming exponential moment).
\end{abstract}

\maketitle

\section{Introduction}\label{S:intro}

\Polya{} urns with an infinite set $\cS$ of possible colurs of the balls
have been studied by
\citet{BTrate, BTinfty,BTnew} and \citet{MM}.
We consider here two special types of \Polya{} urns with the colour space
$\cS=\bbR^d$ ($d\ge1$), which 
both are associated with random walks on $\bbR^d$.
To distinguish them, we call them
\emph{single ball addition random walk (\SRW) \PU{s}} 
and
\emph{deterministic addition random walk (\DRW) \PU{s}}.
The \DRW{} type is the \Polya{} urn considered in \cite{BTrate, BTinfty},
and it is 
included among the more general urns in \cite{BTnew,MM} and studied further
there. 
The \SRWPU{} 
differs from the urns considered in \cite{BTrate, BTinfty,BTnew,MM}
by having random replacements, but it is closely related to the \DRW{} type.
We begin by giving a brief definition of an \SRWPU{},
and refer to \refS{SPolya} for the \DRW{} type and for 
further details, 
including the connection between the two models,
as well as a definition of more general \Polya{} urns.

For simplicity, we first consider an important special case of 
an \SRWPU. 
In this case, the urn contains a (finite) number of balls, each labelled
with a vector $X_i\in\bbR^d$, and starts at time 0
with a single ball labelled with 0.
Furthermore,
the urn evolves by drawing a ball
uniformly at random from the urn, noting its label, $\hX_n$ say,
and replacing it together with a new ball which is
labelled with $X_{n+1}:=\hX_n+\eta_n$, where $\eta_n$ are \iid{} random
variables with some given distribution in $\bbR^d$.
At time $n\ge0$, this urn contains $n+1$ balls. We describe the composition
(= state) of the urn by the measure  (on $\bbR^d$)
\begin{equation}\label{urn0}
  \mu_n:=\sum_{i=0}^n \gd_{X_i},
\end{equation}
where $\gd_x$ (the Dirac delta) denotes a point mass at $x$, and 
$X_0,\dots,X_n$ are the balls in the urn.

The general \SRW{} \Polya{} urn is
an extension of this model; the evolution proceeds in the same way, 
but the initial number of balls 
may be different from 1; in fact, it may be any real number $\rho>0$, 
and the initial labels may be described by an arbitrary 
measure $\mu_0$ on $\bbR^d$ with $\mu_0(\bbR^d)=\rho\in(0,\infty)$,
see  \refS{SPolya} for details.
To draw a ball from an urn with composition $\mu_n$ means that we pick a
colour with the normalized distribution $\tmu_n$, where we for any non-zero
finite measure $\mu$ on a space $\cS$ define its normalization by
\begin{equation}\label{tmu}
  \tmu:=\mu(\cS)\qw\mu
. \end{equation}
In the case above, where \eqref{urn0} holds, $\tmu_n$ is the
empirical distribution of the sequence of colours $X_0,\dots,X_n$.

We assume that the offsets $\eta_n$ have a finite second moment, \ie,
$\E|\eta|^2<\infty$. (We use $\eta$ to denote a generic offset $\eta_n$.)
Then \citet{BTrate, BTinfty,BTnew}
and \citet{MM} proved,
for the \DRW{} model, 
that the normalized compositions 
$\tmu_n$ are
asymptotically normal. To state this formally, we rescale the distributions,
using the following notation from \cite{MM}.
Let $\cP(\bbR^d)$ be the space of Borel probability measures on $\bbR^d$.
If $a>0$ and $b\in\bbR^d$, let $\Thetaab:\cP(\bbR^d)\to\cP(\bbR^d)$ be the 
rescaling mapping defined by:
\begin{equation}\label{Thetaab0}
  \text{if}\quad X\sim \mu\in\cP(\bbR^d),
\qquad\text{then}\qquad
\frac{X-b}{a}\sim \Thetaab\xpar{\mu}.
\end{equation}
Note that if $\mu_n$ is given by \eqref{urn0}, 
and thus $\tmu_n=\frac{1}{n+1}\sum_0^n \gd_{X_i}$,
then 
rescaling $\tmu_n$ by $\Thetaab$ is the same as rescaling all $X_i$ in the
natural way:
\begin{equation}\label{Thetabmu}
  \Thetaab(\tmu_n) = 
\frac{1}{n+1}\sum_{i=0}^n\gd_{(X_i-b)/a}.
\end{equation}

We regard all vectors as column vectors.
It is proved in \cite{BTrate, BTinfty,BTnew,MM} that, for 
\DRWPU{s},
if $\bm:=\E\eta$, then
\begin{equation}\label{urnpto}
  \Theta_{\sqrt{\log n},\,\bm\log n}(\tmu_n) \pto N\bigpar{0,\E[\eta\eta\tr]},
\end{equation}
with convergence in $\cP(\bbR^d)$
with the usual weak topology;
furthermore \citet{MM} showed that the convergence in \eqref{urnpto} hold
\as{} if   $|\eta|$ has a finite exponential moment.
One of the main purposes of the present paper is to show that \as{}
convergence 
always holds, 
assuming only a second moment,
for both  types of \Polya{} urns associated with random walks considered here.

\begin{theorem}\label{TP}
  Consider an \SRWPU, with 
\iid{} offsets $\eta_n\in\bbR^d$ and 
an initial composition $\mu_0$ that is 
an arbitrary  non-zero finite measure on $\bbR^d$.
Assume that $\E|\eta|^2<\infty$ and let $m:=\E\eta$. 
Then, as \ntoo, 
\begin{equation}\label{tp}
  \Theta_{\sqrt{\log n},\,\bm\log n}(\tmu_n) \asto N\bigpar{0,\E[\eta\eta\tr]},
\end{equation}
in $\cP(\bbR^d)$ with the usual weak topology.

The same result holds also for  \DRWPU{s}.
\end{theorem}
Note that the \rhs{} in \eqref{tp}
is non-random; it is a fixed
distribution in $\cP(\bbR^d)$.
Note also that the variance in the limit in \eqref{tp} is $\E[\eta\eta\tr]$
and not $\Var[\eta]=\E[\eta\eta\tr]-mm\tr$, \cf{} \refE{Edepth}.

\begin{remark}\label{Rinspect}
If we inspect the urn 
by drawing a ball at random from the urn at time $n$
(without interfering with the urn process), 
and let $X^*_n$ be is its
  colour, then, conditionally on the urn composition $\mu_n$,  
the distribution of $X^*_n$ is $\tmu_n$. Hence, 
recalling \eqref{Thetaab0},
\eqref{tp} can also be
written as
\begin{equation}\label{tp2}
  \cL\Bigpar{\frac{X^*_{n}-\bm\log n}{\sqrt{\log n}}\Bigm| \mu_n}
\asto 
N\bigpar{0,\E[\eta\eta\tr]}
\end{equation}
in $\cP(\bbR^d)$. 
We can also rewrite \eqref{tp2} as a conditional convergence in distribution:
\begin{equation}\label{tp3}
\text{Conditioned on $(\mu_n)_1^\infty$, a.s.,}\quad
  \frac{X^*_{n}-\bm\log n}{\sqrt{\log n}}
\dto 
N\bigpar{0,\E[\eta\eta\tr]}.
\end{equation}
\end{remark}

\begin{remark}\label{Rannealed}
By unconditioning in \eqref{tp3}, it follows that
  \begin{equation}\label{tp3a}
  \frac{X^*_{n}-\bm\log n}{\sqrt{\log n}}
\dto 
N\bigpar{0,\E[\eta\eta\tr]}
.
\end{equation}
This is a much simpler result, which \eg{} easily
follows from the correspondence with \rrt{s} used below and
the asymptotic normal distribution of the depth of a random node in a
\rrt, see \cite[Theorem 6.17]{Drmota}, together with
the usual central limit theorem for \iid{} variables.

In the language of statistical physics, 
we study in \eqref{tp2}--\eqref{tp3} the \emph{quenched}
  version of the problem, where we fix a realization of the 
urn process $(\mu_n)$,
and then consider the random variable
  $X^*_n$,
obtaining results for \aex{} realization of the urn process.
The corresponding \emph{annealed} version, where we just consider $X^*_n$ as a
random variable obtained by randomly constructing the urn $\mu_n$ and
choosing a ball $X^*_n$ in it,
is the much simpler \eqref{tp3a}. Note that 
the distribution of $X^*_n$ in
the annealed version is $\E\tmu_n$, the expectation of the random measure
$\tmu_n$ defined above. Hence, \eqref{tp3a} can be written as
\begin{equation}\label{tp1a}
  \Theta_{\sqrt{\log n},\,\bm\log n}(\E\tmu_n) \to
N\bigpar{0,\E[\eta\eta\tr]}
.
\end{equation}
We  can  regarded \eqref{tp3a}--\eqref{tp1a}
as the annealed version of \refT{TP}.
Similar unconditioning to annealed versions can be done in 
the theorems for random trees in Sections \ref{Stree2} and \ref{Sbinary}.
\end{remark}

The proofs by \cite{BTrate, BTinfty,BTnew} and \cite{MM}
are based on a connection
between \Polya{} urns and the \rrt, see \refS{Skorre}.
We do the same in the present paper; we
also introduce a weighted modification of the \rrt{} in order to treat
\Polya{} urns with an arbitrary inital configuration, see \refSS{SSrrt}.
The \SRWPU{s} correspond to branching
random walks on the (weighted) \rrt, and thus \refT{TP} is equivalent to
\as{} convergence of the empirical distribution, suitably normalized, for a
branching random walk on a (weighted) \rrt. Furthermore, as is also
well-known, the \rrt{} can be embedded in the continuous time Yule tree,
and thus the result can be interpreted as \as{} convergence of the
normalized empirical distribution of a branching random walk on a Yule tree.
(This extends to the weighted \rrt{} and a weighted Yule tree defined in
\refSS{SSYule}.) Such \as{} convergence results for branching processes have
been shown, in much greater generality, 
by \eg{} \citet[Theorem 4]{Uchiyama},
and thus \refT{TP} essentially
follows from known results in branching process theory.

One purpose of the present paper is to make this connection explicit, by
stating results for branching random walks on \rrt{s} and Yule trees in a
form corresponding to the \Polya{} urn theorem above, 
including the weighted cases.
We prove these results for random trees
using the standard method of showing convergence of a
suitable martingale of 
functions, used also by \citet{Uchiyama}, 
\citet{Biggins-uniform1d,Biggins-uniform}
and others.
(For this, we use a Sobolev space of functions; see \refR{Rspaces}.)
We give complete proofs, 
both for completeness and
because we want to show how the proofs work in this simple case
where we can give explicit expressions instead of estimates, and
without the distractions caused by the greater generality in
\cite{Uchiyama}, and also because we have not been able to find published
results with precisely the formulations used here, 
including the weighted case.
Furthermore, we give proofs with explicit calculations both for the
(weighted) Yule trees and the \rrt{s}; as said above, the results for these
trees are equivalent, so it suffices to prove one of the
cases. Nevertheless, it is possible to prove the result directly, with explicit
calculations, for both cases, 
and we find it interesting and instructive to do so
and see the similarities and differences between the two cases.

The \PU{s}, random trees and branching random walks that we consider are
defined in Sections \ref{SPolya}--\ref{Sbranch}, and the connection between
them is given in \refS{Skorre}.
The results for \rrt{s} and Yule trees are stated in \refS{Stree2} and
proved in Sections \ref{SpfTYule}--\ref{SpfTrrt}. \refT{TP} above is proved
in \refS{SpfTP}.
\refS{Sbinary} gives analoguous results for \bst{s} and binary Yule trees.
\refS{Sproblem} contains some open problems.

\section{\Polya{} urns}\label{SPolya}

As said in the introduction, we consider a general version of \Polya{} urns,
where we have a measurable space $\cS$ of \emph{colours} (i.e., {types}), 
and the state, or \emph{composition},
of the urn is given by a finite measure $\mu$ on $\cS$.
This version of \Polya\ urns has been introduced in a special case by
\citet{BlackwellMcQ}, 
and in general 
by \citet{BTnew}, see also \cite{BTrate,BTinfty}, and by \citet{MM}. 
Although our main theorem is only for \Polya{} urns of
the special types associated with random walks,
we give the definition of \PU{s} in a general form as in 
\cite{BTnew,BTrate,BTinfty,MM}.
Furthermore, we allow also random replacements, see also \cite{SJ327}.

The interpretation of the measure
$\mu$ describing the state of the
urn is that if $A\subseteq\cS$, then $\mu(A)$ is the
total mass of the colours in $A$. 
The classical case with a finite number of
discrete balls  of assorted colours, 
can be treated by representing
each ball of colour $x$ 
by a point mass $\gd_x$;  
in other words,
if the urn contains $m$ balls with colours
$x_1,\dots,x_m$, then it is represented by the measure
\begin{equation}\label{urn1}
  \mu=\sum_i \gd_{x_i}, 
\end{equation}
and thus $\mu$ is a discrete
measure where $\mu\set x$ is the number of balls of colour $x$.
It has often been remarked that the classical case easily generalizes to
non-integer ``numbers of balls'' of each colour; in the measure formulation
considered here, this means that 
$\mu$ is an arbitrary discrete finite measure. 
The general measure version is a further generalization, where 
$\cS$ may be infinite and $\mu$ 
may be, e.g., a  continuous measure.

To define a measure-valued \Polya{} urn process,
we assume that we are given a colour space $\cS$ (a measurable space).
We let $\cM(\cS)$ denote the space of finite measures on $\cS$, 
let $\cMx(\cS):=\cM(\cS)\setminus\set0$,
and define for
each $\mu\in\cMx(\cS)$ its normalization $\tmu$ as
as the probability measure \eqref{tmu}.
We assume also that we are given a
\emph{replacement rule}, which may be deterministic or random.
In the deterministic case it is a (measurable) function $x\mapsto R_x$
mapping $\cS$ into $\cM(\cS)$; in other words, $R_x$ is a kernel from $\cS$
to itself \cite[p.~20]{Kallenberg}. 
In the random case, each $R_x$ is a random
element of $\cM(\cS)$; 
formally the replacement rule is a (measurable) mapping $x\mapsto\cR_x$
mapping $\cS$ into
the space $\cP(\cM(\cS))$
of probability measures on $\cM(\cS)$, 
\ie, a probability kernel from $\cS$ to $\cM(\cS)$,
but it is convenient to represent each $\cR_x$ by a random $R_x\in\cM(\cS)$
having distribution $\cR_x$.

The \Polya{} urn starts with a given initial composition $\mu_0\in\cMx(\cS)$.
In each step, we ``draw a ball from the urn''; this means that,
given everything that has happened so far, if the current composition 
of the urn  is
$\mu_n$, then we randomly select a colour $X_n$ with distribution $\tmu_n$.
We then ``return the ball together with the replacement $R_{X_n}$'', which
means that we update the state of the urn to
\begin{equation}\label{update}
  \mu_{n+1}:=\mu_n+R_{X_n}.
\end{equation}
In the case when the replacements $R_x$ is random, \eqref{update} should be
interpreted to mean that given $X_n$ and everything that has happened earlier, 
we take a fresh random $R_{X_n}$ with the distribution $\cR_{X_n}$. Thus,
given $X_n$, $R_{X_n}$ is independent of the history of the process.
(It is shown in \cite{SJ327} that an urn with random replacements is
equivalent to an urn with deterministic replacements on the larger colour
space $\cS\times\oi$; we will not use this.)

The update \eqref{update} is repeated an infinite number of times;
this defines the \Polya{} urn process as a Markov process.
The process is well-defined, with every $\mu_n\in\cMx(\cS)$,
since we have assumed that each $R_x$ is a finite measure.
(Thus $R_x$ is non-negative; there are no subtractions of balls in this
version.)

\begin{remark}
  We use the name ``replacement'' to conform with 
\cite{BTrate, BTinfty,BTnew,MM}, although $R_x$ really is an addition to the
urn rather than a replacement, since we also return the drawn balls.
(The real replacement is $\gd_x+R_x$.) A version with a true replacement,
without replacement of the drawn ball, is studied in \cite{MM}, but will not
be considered here.
\end{remark}

One special case, which we call \emph{single ball addition} 
is when each replacement consists of a single ball of a
random colour (with distribution $\cR_x$ depending on the colour $x$ of the
drawn ball as above). In other words, for each $x\in\cS$, 
$R_x$ is a random measure of the
type $R_x=\gd_{r_x}$ for a random variable $r_x\in\cS$.
In this case, 
let $\rr_x:=\cL(r_x)\in\cP(\cS)$  be the distribution of $r_x$ and
define $\mux_n\in\cM(S)$ as the composition 
$\mu_n\cdot\rr$
of $\mu_n$ and the kernel
$x\mapsto \rr_x$ (from $\cS$ to itself), see \cite[p.~21]{Kallenberg}, i.e.
\begin{equation}\label{mux}
  \mux_n(A):=
\int_{\cS}\rr_x(A)\dd\mu_n(x)
=\int_{\cS}\P(r_x\in A)\dd\mu_n(x).
\end{equation}
Note that the total mass $\mux_n(\cS)=\mu(\cS)$ and that therefore
$\tmux_n=\tmu_n\cdot\rr$. Hence, in the \Polya{} process above, given the
present state $\mu_n$, the distribution of the colour of the next ball added
to the urn is $\tmux_n$. Furthermore, if $Y_n$ is this colour, \ie,
$Y_n:=r_{X_n}\in\cS$, then $\mu_{n+1}=\mu_n+\gd_{Y_n}$, and thus, by \eqref{mux},
\begin{equation}\label{muxupdate}
  \mux_{n+1}
=\mu_{n+1}\cdot\rr
=\mu_n\cdot\rr+\gd_{Y_n}\cdot\rr=\mux_n+\rr_{Y_n}.
\end{equation}
This means that $\mux_n$ also is a \Polya{} urn process, as defined above,
with deterministic
replacements $\rr_x$. 
We state this formally.

\begin{lemma}\label{Lux}
Let $(\mu_n)$ be a single ball addition  \Polya{} urn process, with random
replacements $R_x=\gd_{r_x}$. Then, with $\rr_x:=\cL(r_x)$ and
$\mux_n:=\mu_n\cdot\rr$,
the sequence $(\mux_n)$ is a \Polya{} urn process with deterministic
replacements $\rr_x\in\cP(\cS)$.
\qed
\end{lemma}

The \Polya{} urns studied in \cite{BTrate, BTinfty,BTnew,MM}
have deterministic replacements that furthermore are 
probability measures; hence these urns are of the type 
\eqref{muxupdate}, 
and \refL{Lux} shows that,
provided the initial value is of the type $\mu_0\cdot\rr$, 
these urns correspond to urns $\mu_n$ with random single ball additions; more
precisely they are given by $\mu_n\cdot\rr$.

\begin{example}
The urns studied in \citet{BlackwellMcQ} have the special form $R_x=\gd_x$;
hence they are single ball addition \Polya{} urns
where the added ball has the same
colour as the drawn one, just as for the original (two-colour)  urns in  
\cite{Markov1917,EggPol1923,Polya1930}. In this case $r_x=x$ and
$\mux_n=\mu_n$, so there is no difference between the two \Polya{} urns in
\refL{Lux}.   
\end{example}

\begin{example}\label{ERW}
The \emph{\SRWPU s} discussed in \refS{S:intro} 
are  a special case of the single
ball addition case,  where $\cS=\bbR^d$ and the replacements are
  translation invariant, i.e., $r_x\eqd x+r_0$ for all $x\in\bbR^d$.
In other words, with $\eta:=r_0$, if we draw a ball of colour $X_n$, it is
replaced together with a ball of colour $X_n+\eta_n$, where $(\eta_n)$ are
independent copies of the random variable $\eta\in\bbR^d$ (with $\eta_n$
independent of $X_n$).

By \refL{Lux}, an \SRWPU{} corresponds to an urn with
colour space $\bbR^d$ and deterministic replacements $\rr_x=\cL(x+\eta)$.
This is the type of urns studied in \cite{BTrate,BTinfty}; they are also
studied in \cite{BTnew,MM} together with more general ones.
We call such urns \emph{\DRWPU{s}}.

Note that for a \DRWPU{}, the translation invariance of $\rr$ shows that the
relation $\mux_n=\mu_n\cdot\rr$ in \refL{Lux}
can be written as a convolution
\begin{equation}\label{*}
  \mux_n = \mu_n * \rr_0 = \mu_n*\nu.
\end{equation}
%
\end{example}

\section{Random trees}\label{Stree1}

The random trees that we study are (mostly) well-known; see for example
\cite{Drmota} and \cite{AN,Yule}.
For convenience, we collect their definitions here.

The trees that we are interested in grow (randomly) in either discrete or
continuous time; 
we thus consider either an increasing sequence of random
trees $T_n$ with an integer parameter $n\ge0$,
or an increasing family $\ctt$ of random trees with
a real parameter $t\ge0$. 
(We use different fonts for the two cases; this will be convenient
to distinguish them in \eg{} the proof of \refT{Trrt} where we consider
trees of both types simultaneously, but has otherwise no significance.) 

For a tree $T$, we let $|T|$ denote its number of nodes; however, when we
consider weighted trees, we instead let $|T|$ denote the total weight, i.e.,
the sum of the weights of the nodes.

\subsection{Trees growing in discrete time}\label{SSrrt}

The \emph{random recursive tree} $T_n$ is constructed recursively. $T_0$ is
just a root. Given $T_n$, we obtain $T_{n+1}$ by adding a new node and
choosing its parent uniformly at random from the already existing nodes.
(We have chosen a notation where $T_n$ has $n+1$ nodes; this is of course
irrelevant for our asymptotic results.)

We consider also a generalization of the \rrt{} that we call 
a \emph{\wrrt}; this is characterized by a parameter $\rho>0$ (the
\emph{weight}). The definition is as for the \rrt, but we give the 
root  weight $\rho$ and every other node weight 1, and when we add a node,
its parent is chosen with probability proportional to its weight.
In other words, when adding a new node to $T_n$, its parent is chosen to be the
root $o$ with probability $\rho/(n+\rho)$, and  to be  $v$
with probability $1/(n+\rho)$, for each of the $n$  existing nodes $v\neq o$.
Note that taking the weight $\rho=1$ gives the  \rrt.

The \emph{\bst} is defined by a similar recursive procedure, but we now
have two types of nodes, \emph{internal} and \emph{external}.
$T_0$ consists of a single external node (the root).
The tree evolves by choosing an external node uniformly at random, and then
converting it to an internal node and adding two new external nodes as
children to it. (One child is labelled left and the other right.) 
Thus $T_n$ has $n$ internal nodes and $n+1$ external nodes; an internal node
has 2 children, and an external node has 0.
Depending on the circumstances, one might either be interested in the
complete tree $T_n$ with $2n+1$ nodes, or just the internal subtree $\ti_n$
with the $n$ internal nodes.

\subsection{Trees growing in continuous time}\label{SSYule}

The \emph{Yule tree} $\ctt$, $t\ge0$,
is the family tree of 
the \emph{Yule process},
which is
a simple Markov continuous-time 
branching process
starting with a single node (= individual) at time 0 and such that every
node lives for ever and gets children according to a Poisson process
with intensity 1.

By symmetry and lack of memory, it is obvious that if $\tau_n$ is the
stopping time  
\begin{equation}
  \label{taun}
\tau_n:=\min\set{t:|\ctt|=n+1}, 
\end{equation}
then the sequence $T_n:=\cT_{\tau_n}$
is a sequence of random recursive trees.

Corresponding to the \wrrt{} above, we define also a 
\emph{weighted Yule tree},
where the root (the initial node) has weight $\rho>0$ and every other
node has weight 1, and each node gets children with intensity 
equal to its weight. (Thus, only the initial node is modified.)
For the weighted Yule tree, we  modify \eqref{taun} and  define $\tau_n$ as
the first time that the 
total weight is $n+\rho$; 
then $T_{\tau_n}$ is a \wrrt{} with the same weight $\rho$.
Note that if $\rho$ is an integer, then the weighted Yule tree can be
obtained by taking $\rho$ independent Yule trees and merging their roots.

Many authors use a  different version of the Yule tree, which we
call the \emph{binary Yule tree} to distinguish the two versions.
The difference is that each individual lives a random time with an
exponential distribution $\Exp(1)$ with rate 1, and that each individual
gets 2 children when she dies. (We do not define any weighted version.)
It is obvious that the number of living individuals follows the same
branching process (the {Yule process}) for both versions, but that the
trees $\ctt$, which contain both the living and dead individuals, will be
different. 
In fact, if we now let $\tau_n$ be the first time that the tree has $n$ dead
individuals, and thus $n+1$ living ones, it is easy to see, again
because of the lack of memory, that the sequence $T_{\tau_n}$ defines a
\bst,
where the dead individuals are internal nodes and the living individuals are
external nodes.

We shall use some simple facts from branching process theory.

First, it is a well-known fact \cite[Theorems III.7.1--2]{AN}
that
for the Yule tree 
(and much more generally), 
\begin{equation}\label{winston}
  |\ctt|/e^t \asto W>0
\end{equation}
for some random variable $W$. (In fact, for the Yule tree,
$W\sim\Exp(1)$, but we do not need this.)
%
For the weighted Yule tree, every child of the root starts an independent
Yule tree, and it follows easily that \eqref{winston} holds in this case too.
(Furthermore, $W$ then has the Gamma distribution  $W\sim\Gamma(\rho)$.)

Taking $t=\tau_n$ in \eqref{winston} yields, 
for a general weight $\rho$,
\begin{equation}
  \frac{n+\rho}{e^{\tau_n}}\asto W
\end{equation}
and thus, \as,
\begin{equation}\label{tauas}
  \tau_n = \log(n+\rho) - \log W + o(1) 
= \log n + O(1).
\end{equation}
(Here and below, the implicit constant in $O(1)$ may be random.)

We note a standard fact.

\begin{lemma} \label{L3a}
If\/ $\ctt$ is the weighted Yule tree, for any $\rho>0$,
then for every $t<\infty$ and $r<\infty$,
$\E|\ctt|^r<\infty$.
\end{lemma}
\begin{proof}
It is well-known that $\E|\ctt|^r <\infty$ for
the standard Yule tree with $\rho=1$,
see  \cite[Corollary III.6.1]{AN}.

For general $\rho$, we may by monotonicity (in $\rho$) assume that $\rho$ is
an integer, and the result then follows by regarding the tree as a union of
$\rho$ independent Yule trees.
\end{proof}

\section{Branching random walks on trees}\label{Sbranch}

Given a rooted tree $T$ and a probability distribution $\nu$ on $\bbR^d$,
a \emph{branched random walk} on $T$, with \emph{offset distribution} $\nu$,
is a stochastic process $(X_v)_{v\in T}$ indexed by the nodes of $T$
that is defined recursively as follows:
\begin{PXenumerate}{BW}
\item \label{BWeta}
Let $\eta_v$, $v\in T$, be \iid{} random variables with $\eta_i\sim\nu$.
\item \label{BW0}
$X_o:=0$, where $o$ is the root of $T$.
\item \label{BWX}
If $w$ is a child of $v$, then $X_w:=X_v+\eta_w$.  
\end{PXenumerate}

In other words, if $v\prec w$ means that $v$ is an ancestor of $w$,
\begin{equation}\label{X}
  X_v:=\sum_{o\prec u\preceq v}\eta_u.
\end{equation}

The tree $T$ is usually random; we then tacitly assume that the random
variables $\eta_v$ are independent of the tree $T$.

\begin{remark}\label{R1}
Alternatively (and equivalently),
we may start with a rooted tree $T$ and 
random variable $\eta$,
by taking $\nu:=\cL(\eta)$,
the distribution of $\eta$;   \ref{BWeta} then says that
$\eta_i$ are independent copies of $\eta$.
In this setting, $\eta$ is called the \emph{offset}.
In the sequel, we use $\eta$ in this sense, to denote a generic random
variable with distribution $\nu$.
\end{remark}

\begin{remark}\label{R2}
We never use $\eta_o$, and may thus ignore it. 
For $v\neq o$, we may think of $\eta_v$ as
associated to the edge leading to $v$ from its parent; then $X_v$ is the sum
of these values for all edges on the path between $o$ and $v$.
(Alternatively, we could change the definition and let $X_o:=v_o$; this
would not affect our asymptotical results.)
\end{remark}

We are interested in the empirical distribution of the variables
$X_v$; this is by definition the random probability measure on $\bbR$
(or $\bbR^d)$
defined by
\begin{equation}\label{bmu}
\bmu:=\frac{1}{|T|}\sum_{v\in T}\gd_{X_v},  
\end{equation}
where $\gd_x$ (the Dirac delta) denotes a point mass at $x$, and $|T|$ is
the number of nodes in $T$. In other words,  given $(X_v)_{v\in T}$, $\bmu$ is
the distribution $\cL(X_V)$ of the value $X_V$ seen at a uniformly randomly
chosen node $V\in T$.

For a weighted tree, we modify the definition \eqref{bmu} by
counting each node $v$ according to its weight $\ww_v$. 
In our cases, $\ww_v=1$ for every $v\neq o$, and thus,
recalling that $|T|$ denotes the
total weight, 
\begin{equation}\label{bmuw}
\bmu:=\frac{1}{|T|}\sum_{v\in T}\ww_v\gd_{X_v}
=\frac{1}{|T|}\Bigpar{\rho\gd_{X_o}+\sum_{v\neq o}\gd_{X_v}}.
\end{equation}
Similarly, the random node $V\in T$ is chosen with probability
proportional to its weight, and then still $\bmu=\cL(X_V)$.
Since only the root has a weight different from 1, it is obvious that the
asymptotic results below are not affected be these modifications,
and the results hold for both definitions \eqref{bmu} and \eqref{bmuw}.
However, the modifications are natural, and convenient in the proofs below,
so we shall use \eqref{bmuw} in the weighted case.

In analogy with the \Polya{} urns defined earlier, we also define the
unnormalized measure
\begin{equation}\label{muw}
\mu:=|T|\bmu=\sum_{v\in T}\ww_v\gd_{X_v}.
\end{equation}

We consider an increasing sequence or family of (random)
trees $T_n$ with an integer parameter $n\ge0$,
or alternatively an increasing family $\ctt$ of random trees with
a real parameter $t\ge0$. 
We consider also
an \iid{} family $(\eta_v)_v$ 
of offsets, defined for all 
$v\in T_\infty:=\bigcup_n T_n$ [or $\bigcup_t \ctt$].
(Thus $\eta_v$ is defined for all $v\in T_n$ [$\ctt$], but does not
depend on the parameter $n$ [$t$].)
We denote the empirical distribution by $\bmu_n$ [$\bmu_t$], 
and 
our goal is to show
that it, suitably rescaled, \as{} converges to
a normal distribution as \ntoo{} or \ttoo; 
see \refS{Stree2} for precise statements.
As in \refT{TP}, this is a question of convergence of a
random probability measure  in the space $\cP(\bbR^d)$ of probability
measures on $\bbR^d$ with the standard (weak) topology, \cf{} \refR{Rannealed}.

We give a simple lemma that will be used later.

\begin{lemma} \label{L3b}
If\/ $\ctt$ is a weighted Yule tree and $\E|\eta|^2<\infty$, then for every
$t<\infty$, 
$\E \bigpar{\sum_{v\in\ctt}|X_v|}^2<\infty$.
\end{lemma}
\begin{proof}
  Trivially, $|X_v|\le\sum_{w\in\ctt}|\eta_w|$ for every $v\in \ctt$ and thus
  \begin{equation}
\sum_{v\in\ctt}|X_v|\le |\ctt|\sum_{w\in\ctt}|\eta_w|.
  \end{equation}
It follows, using the \CSineq, that
\begin{equation}
  \E\Bigpar{\Bigpar{\sum_{v\in\ctt}|X_v|}^2\mid\ctt}
\le |\ctt|^2\sum_{u,w\in\ctt}\E(|\eta_u|\,|\eta_w|)
\le |\ctt|^4\E|\eta|^2
\end{equation}
and thus, using \refL{L3a},
\begin{equation}
  \E\Bigpar{\sum_{v\in\ctt}|X_v|}^2
\le \E|\eta|^2 \E|\ctt|^4 <\infty.
\end{equation}
\end{proof}

\section{\Polya{} urns and trees}\label{Skorre}

The proofs 
by \citet{BTrate, BTinfty,BTnew} and \citet{MM}
are based on a natural
connection between 
\Polya{} urns and branching Markov chains on  \rrt{s},
and in particular between
\DRW{} \Polya{} urns and branching random walks
on \rrt{s}. 
In our setting, 
we use two  versions, one for 
\SRW{} urns and one for
\DRW{} urns, and we include also the weighted case (when $\mu_0(\cS)\neq1$).

\subsection{\SRWPU{s}}\label{SSkorr1} 
Consider a \Polya{} urn of the single ball addition type, see \refS{SPolya};
we let the initial composition $\mu_0$ have arbitrary mass $\rho>0$, but
assume that it is concentrated at 0; thus, $\mu_0=\rho\gd_0$.
Regard this initial mass as a ball with weight $\rho$ and colour $0$;
let all balls added later to the urn have weight 1.
Regard the balls in the urn as nodes in a tree, where the initial ball is
the root and
each new ball added
after drawing a ball becomes a child of the drawn ball.
It is obvious that the resulting random tree process is the weighted \rrt{}
defined in \refSS{SSrrt}. Furthermore, if the \Polya{} urn is of the 
\SRW{} type, then the labels on the ball form a branching random walk
\eqref{X}, with the same offset $\eta$; note that the measures $\mu_n$ and
$\tmu_n$ defined earlier (Sections \ref{SPolya} and \ref{Sbranch})
are the same for the \Polya{} urn and the branching random walk.

This means that, at least when $\mu_0=\rho\gd_0$, the first part of
\refT{TP} is equivalent to a result for the \rrt, see \refT{Trrt} below
and \refS{SpfTP}.

\subsection{\DRWPU{s}}\label{SSkorrD}
For the \DRWPU{s}, we describe the
connection used by 
\cite{BTrate, BTinfty,BTnew,MM} as follows.
Consider first an arbitrary measure-valued \Polya{} urn 
with deterministic
replacements $R_x\in\cM(\cS)$.
Denote the successive additions to the urn by,  see \eqref{update},
\begin{align}
\gD\mu_{n+1}:=\mu_{n+1}-\mu_n=R_{X_n},
\end{align}
and let $\gD\mu_0:=\mu_0$, the initial composition.
We may pretend that the different additions $\gD\mu_k$, $k\le n$, 
are identifiable
parts of $\mu_n$. Hence when we draw a ball, we can do it in two steps; we
first select an index $k\le n$, with probability $\gD\mu_k(\cS)/\mu_n(\cS)$, 
and then, given $k$, select $X_n$ with distribution $\widetilde{\gD\mu_k}$.
This defines a growing family $T_n$ of trees, where $T_n$ has node set
\set{0,\dots,n}, and $T_{n+1}$ is obtained from $T_n$ by adding $n+1$ as a
new node with mother $k$, the index selected when choosing $X_n$ in the
construction of the \Polya{} urn.

From now on, we assume that all replacements $R_x$ are probability measures,
\ie, have mass $R_x(\cS)=1$. Then it is obvious that this random family of
trees $T_n$ is the weighted \rrt{} with weight $\rho=\mu_0(\cS)$.
We mark each node $n$ in this tree by the measure $\gD\mu_n$, and also, for
$n>0$, 
by the colour $X_{n-1}$ of the ball drawn to find this addition; we write
$Z_n:=X_{n-1}$.
Then, given the trees $T_n$, $n\ge0$, these marks, and thus the \Polya{} urn,
are defined recursively as follows, with $o=0$, the root of the tree, 
\begin{PXenumerate}{P}
\item \label{PU0}
$\gD\mu_o:=\mu_0$
\item \label{PU}
If a node $v\neq o$ has mother $u$, then $Z_v$ is drawn with the
distribution $\gD\mu_u$, and then $\gD\mu_v:=R_{Z_v}$.
\end{PXenumerate}

If we further specialize to a \DRWPU, then,
see \refE{ERW}, for $v\neq o$,  assuming $\eta$ to be independent of 
everything else,
\begin{equation}\label{gdmu}
\gD\mu_v:=R_{Z_v}=\rr_{Z_v}=\cL(\eta+Z_v\mid Z_v).
\end{equation}

\begin{remark}
  It is easily seen that the relations in Sections \ref{SSkorr1} and
  \ref{SSkorrD} are connected by the correspondence in \refL{Lux};
if we take the tree in \refS{SSkorr1} and mark each node with colour $x$ by
$R_x$ ($\rho R_0$ for the root), then we obtain the corresponding
process in \refSS{SSkorrD}.
\end{remark}

\section{Results for trees}\label{Stree2}

We state here the results for the (weighted) \rrt{} and Yule tree; 
proofs are given in Sections \ref{SpfTYule}--\ref{SpfTrrt}. 
As said in the introduction, the results for Yule trees are essentially
proved by \citet[Theorem 4]{Uchiyama}; the weighted Yule trees considered
here are 
not quite included in his conditions (which otherwise are very general),
but his result is easily extended to the present case.
See also  \refS{Sbinary}, where
similar results for the \bst{} and the binary Yule tree are given.

Recall the definitions of the random trees in \refS{Stree1},
the empirical measures $\bmu_n$ or $\bmu_t$ in \refS{Sbranch},
and $\Thetaab$ in \eqref{Thetaab0}.

We assume that $\E|\eta|^2<\infty$, and let as above
$  \bm:=\E \eta \in \bbR^d$.
Convergence in the space $\cP(\bbR^d)$ of probability measures
is always in the
usual weak topology.
\begin{theorem}
  \label{Trrt}
Let\/ $T_n$ be the \rrt{} and
suppose that\/ $\E|\eta|^2<\infty$.
Then, as \ntoo,
in $\cP(\bbR^d)$,
\begin{equation}\label{trrt1}
  \Theta_{\sqrt{\log n},\,\bm\log n}(\bmu_n) \asto
N\bigpar{0,\E[\eta\eta\tr]}
.
\end{equation}

More generally, the same result holds for the \wrrt{} with an arbitrary
weight $\rho>0$ defined in \refSS{SSrrt}.
\end{theorem}

\begin{theorem}[essentially \citet{Uchiyama}]
  \label{TYule}
Let\/ $\ctt$ be the Yule tree and
suppose that\/ $\E|\eta|^2<\infty$.
Then, as \ttoo,
in $\cP(\bbR^d)$,
\begin{equation}\label{tyule}
  \Theta_{\sqrt{t},\,\bm t}(\bmu_t) \asto
N\bigpar{0,\E[\eta\eta\tr]}
.
\end{equation}

More generally, the same result holds for the weighted version 
with an arbitrary weight $\rho>0$ defined in \refSS{SSYule}.
\end{theorem}

\begin{remark}
  \label{Ralt}
As for \refT{TP}, the results can be stated as 
conditional convergence in distribution, see  \refR{Rinspect}.
Let $V_n$ be a random node in the \rrt{} $T_n$, as in \refS{Sbranch} chosen with
probability proportional to its weight (and thus uniformly when $\rho=1$).
Then \eqref{trrt1} is equivalent to
\begin{equation}\label{trrt2}
  \cL\Bigpar{\frac{X_{V_n}-\bm\log n}{\sqrt{\log n}}\Bigm| T_n,\set{\eta_v}_v}
\asto 
N\bigpar{0,\E[\eta\eta\tr]}
,
\end{equation}
which can be written
\begin{equation}\label{trrt3}
\text{Conditioned on \set{T_n} and \set{\eta_v}, a.s,}\quad
  \frac{X_{V_n}-\bm\log n}{\sqrt{\log n}}
\dto 
N\bigpar{0,\E[\eta\eta\tr]}.
\end{equation}
The same applies to \refT{TYule}, and the binary trees in
\refS{Sbinary}; we leave the details to the reader.

Note also that by unconditioning in \eqref{trrt2}--\eqref{trrt3}, we obtain
a (simpler) annealed version, \cf{}
\refR{Rannealed}.
\end{remark}

\begin{example}\label{Edepth}
  As a special case of the results above,  let $\eta\equiv1$
  (deterministically). Then $X_v$ is the depth of $v$, and thus Theorems
  \ref{Trrt} and \ref{TYule}  show that the distribution of node depths in a
(weighted)   \rrt{} or a Yule tree \as{} is asymptotically normal.
Note that in this case,  with the normalizations above,
the limit is $N(0,1)$.
(This is known, at least in the unweighted case,
for example from \cite{Uchiyama} or \cite[Remark 6.19 and (6.25)]{Drmota};
see also \cite[Theorem 6.17]{Drmota} on the insertion depth, 
which is related to the annealed version,
and the corresponding result for binary search trees and binary Yule
trees
in \cite{ChauvinDJ} and \cite{ChauvinKMR}.)

This special case 
implies that in the asymptotic variance $\E[\eta\eta\tr]$ in the theorems
above, we can interpret $\bm\bm\tr$ as coming from the random fluctuations
of the depths;
thus the contribution coming from the fluctuations of the offsets is
$\E[\eta\eta\tr]-\bm\bm\tr$, which is the covariance matrix of $\eta$.
\end{example}

\begin{remark}
  The same problem for conditioned Galton--Watson trees, 
which includes for example uniformly
  random plane trees and binary trees, has been studied by
\citet{AldousISE}; see also, \eg, 
\cite{SJ164}.
The results for those random trees are very different from the present ones,
with convergence in distribution to a non-random limit (known as ISE).
\end{remark}

\section{Proof of \refT{TYule}}\label{SpfTYule}

In this section we prove \refT{TYule}; we then show in \refS{SpfTrrt}
that \refT{Trrt} is a simple consequence of \refT{TYule}. 
On the other hand, it is also easy to prove \refT{Trrt} directly using same
arguments as for \refT{TYule} with only minor modifications; we give in
\refS{SpfTrrt} also this, alternative, proof for comparison. (This is the
method used by \citet{MM}, under somewhat stronger conditions.)

 The basic idea is the same as in \citet{Uchiyama}
(and in many other papers), although
the details are different from the very general case in \cite{Uchiyama};
we construct a
 martingale from the characteristic function (Fourier transform) of the
 empirical distribution, and then use 
 martingale theory to obtain \as{} uniform convergence of this martingale,
 leading to convergence of the \chf{} of suitable rescaled $\Thetaab(\bmu)$.
See further \refR{Rspaces}.

We consider a branching random walk on
the Yule tree $\ctt$ with offsets $\eta_v$ as described in Sections
\ref{Stree1}--\ref{Sbranch}.
In this section, we assume until further notice (at the end of the
section) that $d=1$. 
(Actually,  all formulas except \eqref{kg} extend with at most notational
differences to $d>1$, but we do not use this; see also \refR{Rspaces}.)
We assume also, for simplicity, that $\rho=1$, so
that we consider the standard Yule tree; the minor modifications for a
general $\rho$ are discussed in \refR{Rrho} after the lemmas.

Let $\cF_t$ be the \gsf{} generated by all events (births and offsets) 
up to time $t$.

$C$ denotes positive constants that may vary from one occurrence to the
next.
They may depend on the offset distribution, but not on $n$, $t$ or other
variables. 

Denote the characteristic function of the offset distribution by
\begin{equation}\label{gf}
  \gf(s):=\E e^{\ii s \eta}.
\end{equation}
Fix throughout the proof $\gd>0$ such that $\Re \gf(s)\ge \frac{3}4$ when 
$|s|\le\gd$, and let $J$ be the interval $[-\gd,\gd]$. 

Define the Fourier transform of a measure $\mu\in\cM(\bbR)$ by 
\begin{equation}\label{FT}
  \hmu(s)=\int_{\bbR} e^{\ii sx}\dd\mu;
\end{equation}
recall that this is the \chf{} if $\mu$ is a probability measure.
Define the complex-valued random function, with $\mu_t$ given by \eqref{muw}
for $\ctt$,
\begin{equation}
  \label{F}
F_t(s):=
\hmu_t(s)
=
\sum_{v\in \ctt} e^{\ii s X_v}, 
\qquad s\in \bbR.
\end{equation}
Note that $F_t(0)=|\ctt|$, and that
 $F_t(s)/F_t(0)$ is the characteristic function of the probability measure
$\bmu_t$, see \eqref{fle} below.

Although $F_t(s)$ is defined for all real $s$, we shall mainly consider $F_t$
as a function on $J$.
We begin by computing the first and second moments of $F_t(s)$.
\begin{lemma}
  \label{L1}
  \begin{thmenumerate}
  \item \label{L1F}
    For every $t\ge0$ and $s\in\bbR$,
    \begin{equation}\label{EF}
      \E F_t(s) = e^{t\gf(s)}.
    \end{equation}
\item \label{L1FF}
    For every $t\ge0$ and $s_1,s_2\in J$,
    \begin{equation}\label{EFF}
      \begin{split}
      \E \bigpar{F_t(s_1)F_t(s_2)} 
&= 
\frac{\gf(s_1)+\gf(s_2)}{\gf(s_1)+\gf(s_2)-\gf(s_1+s_2)}
e^{t(\gf(s_1)+\gf(s_2))}
\\&\qquad{}-
\frac{\gf(s_1+s_2)}{\gf(s_1)+\gf(s_2)-\gf(s_1+s_2)}
e^{t\gf(s_1+s_2)}
.
      \end{split}
    \end{equation}
  \end{thmenumerate}
\end{lemma}

\begin{proof}
\pfitemref{L1F}
Each existing node $v$ gets a new child, $w$ say,  with intensity 1, 
independently of the past. 
If this happens, then $F_t(s)$ increases by $e^{\ii s X_w}=e^{\ii s(X_v+\eta_w)}$.
Since $\eta_w$ independent of $\cF_t$, 
the conditional expectation given $\cF_t$ of this possible jump of $F_t(s)$
is 
$\E\bigpar{e^{\ii s(X_v+\eta_w)}\mid\cF_t}=e^{\ii s X_v}\gf(s)$.
It follows by standard Poisson process theory that, 
for any fixed $s\in\bbR$,
there exists a martingale $M'_t$  (depending on $s$) such that,
in the notation of stochastic calculus, 
\begin{equation}\label{stoc}
  \dd F_t(s) =\sum_{v\in\ctt} e^{\ii s X_v}\gf(s)\dd t + \dd M'_t
=\gf(s)F_t(s)\dd t + \dd M'_t.
\end{equation}
In particular, taking the expectation, we obtain 
\begin{equation}
  \label{er66}
  \ddt \E F_t(s) =  \gf(s) \E F_t(s).
\end{equation}
This differential equation, with the initial value $\E F_0(s)=1$, has the
solution \eqref{EF}.

\pfitemref{L1FF}
Arguing as for part \ref{L1F}, we obtain for every real $s_1,s_2$
\begin{align}\label{erII}
&\ddt \E \bigpar{F_{t}(s_1)F_t(s_2)}
\notag\\&\quad
= \E\sum_{v\in\ctt} 
\Bigpar{
\bigpar{F_t(s_1)+e^{\ii s_1(X_v+\eta_w)}}  \bigpar{F_t(s_2)+e^{\ii
    s_2(X_v+\eta_w)}}
-F_t(s_1)F_t(s_2)}
\notag\\&\quad
=\E\Bigpar{F_t(s_1)F_t(s_2)e^{\ii s_2 \eta_w} + F_t(s_1)F_t(s_2)e^{\ii s_1 \eta_w}
+ F_t(s_1+s_2)e^{\ii (s_1+s_2) \eta_w}}
\notag\\&\quad
=\bigpar{\gf(s_1)+\gf(s_2)}\E\bigpar{F_t(s_1)F_t(s_2)}
+ \gf(s_1+s_2)\E F_t(s_1+s_2).
\end{align}
For $s_1,s_2\in J$, we have 
\begin{equation}
  \label{boss}
\Re\bigpar{\gf(s_1)+\gf(s_2)-\gf(s_1+s_2)}\ge \frac{3}4+\frac{3}4-1=\frac12>0
\end{equation}
and thus \eqref{erII} has the solution \eqref{EFF}, again recalling the
initial valure $F_0(s)=1$.
\end{proof}

The proof is based on the following martingale.
\begin{lemma}
\label{LM}
Let 
\begin{equation}\label{M}
  M_t(s) := \frac{F_t(s)}{\E F_t(s)} = e^{-t\gf(s)}F_t(s).
\end{equation}
Then $M_t(s)$, $t\ge0$, is a (complex) martingale for every fixed $s\in\bbR$.
Furthermore, for $s\in J$, this martingale is uniformly $L^2$-bounded:
\begin{equation}\label{lm}
  \E |M_t(s)|^2\le C,
\qquad t\ge0,\, s\in J.
\end{equation}
\end{lemma}

\begin{proof}
The first part is standard: it follows 
from \eqref{stoc} above that
$e^{-\gf(s)t}F_t(s)$ is a martingale.

Furthermore, \eqref{M} and \eqref{EFF} yield, for $s_1,s_2\in J$,
    \begin{equation}\label{OFF}
      \begin{split}
      \E \bigpar{M_t(s_1)M_t(s_2)} 
&= 
\frac{\gf(s_1)+\gf(s_2)}{\gf(s_1)+\gf(s_2)-\gf(s_1+s_2)}
\\&\qquad{}-
\frac{\gf(s_1+s_2)}{\gf(s_1)+\gf(s_2)-\gf(s_1+s_2)}
e^{t(\gf(s_1+s_2)-\gf(s_1)-\gf(s_2))}
.
      \end{split}
    \end{equation}
Using \eqref{boss}, 
this yields the estimate, still for $s_1,s_2\in J$ and all $t\ge0$,
    \begin{equation}\label{YFF}
      \begin{split}
\bigabs{      \E \bigpar{M_t(s_1)M_t(s_2)} }
&\le 
2\bigabs{\gf(s_1)+\gf(s_2)}+
2\bigabs{\gf(s_1+s_2)}
\le 6
.
      \end{split}
    \end{equation}
Furthermore, $F_t(-s)=\overline{F_t(s)}$ and thus
$M_t(-s)=\overline{M_t(s)}$;
hence \eqref{lm} follows by taking $s_1=-s_2=s$ in \eqref{YFF}.
\end{proof}

\begin{remark}
The proof shows also that, typically, $M_t(s)$ is not $L^2$-bounded for every
real $s$, see \eqref{OFF}, which holds as soon as the denominator 
$\gf(s_1)+\gf(s_2)-\gf(s_1+s_2)\neq0$.
Hence it is necessary to restrict to some
interval $J$.
(Our choice of $J$ is not the largest possible, 
but it is convenient for our purposes.)
\end{remark}

\refL{LM} implies that for every fixed $s\in J$,
the martingale $M_t(s)$ converges a.s.
A crucial step is to improve this to uniform convergence for all $s\in J$,
i.e., \as{} convergence of $M_t$ as an element of the Banach space $C(J)$.
However, we shall not work in $C(J)$, since we find it difficult to estimate 
the first or second moment of
$\norm{M_t}_{C(J)}=\sup_{s\in J}|M_t(s)|$ directly; another technical
problem is that $C(J)$ does not have the \RNP{} (see below).
Instead we use the Sobolev space $\WWJ$ defined by
\begin{equation}\label{WWJ}
  \WWJ:=\bigset{f\in L^2(J):f'\in L^2(J)},
\end{equation}
with the norm
\begin{equation}\label{normWWJ}
  \normWWJ{f}^2:=\norm{f}_{L^2(J)}^2+\norm{f'}_{L^2(J)}^2.
\end{equation}

\begin{remark}
  The definition \eqref{WWJ}
has to be interpreted with some care, since $f$ in general
  is not differentiable everywhere.
The general definition of Sobolev spaces in several variables 
\cite{Adams,Mazja}
uses
distributional (weak) derivatives. In the present one-variable case,
we can just require that $f$ is absolutely continuous on $J$, 
so that $f'$ exists
a.e.\ in $J$ in the usual sense, and then assume $f,f'\in L^2(J)$.
Equivalently, 
$\WWJ=\bigset{f:f(x)=f(0)+\int_0^x g(y)\dd y \text{ for some }g\in L^2(J)}$.
\end{remark}

Then $\WWJ$ is an Hilbert space.
Furthermore, there is a continuous inclusion $\WWJ\subset C(J)$, 
and thus an estimate
\begin{equation}\label{kg}
  \norm{f}_{C(J)}\le C \normWWJ{f},
\qquad f\in\WWJ.
\end{equation}
This is a special case of the Sobolev embedding theorem
\cite[Theorem 5.4]{Adams}; in the present
(one-variable) case, it is an easy consequence of the \CSineq, which implies
that if
$f\in \WWJ$ and $[a,b]\subseteq J$, then
\begin{equation}\label{cc}
  |f(b)-f(a)|\le \int_a^b|f'(x)|\dd x
\le (b-a)\qq \norm{f'}_{L^2[a,b]}
\le (b-a)\qq\normWWJ{f}
.
\end{equation}
Hence $f\in C(J)$, and \eqref{kg} follows easily from \eqref{cc}.

The (random) function $F_t$ defined in \eqref{F} is an infinitely differentiable
function of $s$. Furthermore, since we assume $\E|\eta|^2<\infty$, the
\chf{} $\gf(s)$ is twice continuously differentiable; hence so is $\E
F_t(s)$ by \eqref{EF} and $M_t(s)$ by \eqref{M}. In particular, 
$M_t\in\WWJ$ for every $t\ge0$.

\begin{lemma}\label{LMWWJ}
  $(M_t)_{t\ge0}$ is a right-continuous $L^2$-bounded martingale in $\WWJ$.
\end{lemma}
\begin{proof}
We begin by estimating the norm.  
By \eqref{normWWJ} and Fubini's theorem,
\begin{equation}\label{emma}
  \begin{split}
\E  \normWWJ{M_t}^2
&=
\E  \normLLJ{M_t}^2 
+  \E\normLLJ{M_t'}^2
\\&
=\E\int_J |M_t(s)|^2\dd s + \E\int_J \Bigabs{\dds M_t(s)}^2\dd s
\\&
=\int_J \E|M_t(s)|^2\dd s + \int_J \E\Bigabs{\dds M_t(s)}^2\dd s
.
  \end{split}
\end{equation}
We know that $\E|M_t(s)|^2$ is bounded by \eqref{lm}, but it remains to
estimate the last integral.

Denote the \rhs{} of \eqref{OFF} by $h(s_1,s_2;t)$, so 
$\E\bigpar{M_t(s_1)M_t(s_2)}=h(s_1,s_2;t)$ when $s_1,s_2\in J$.
Taking partial derivatives with respect to both $s_1$ and $s_2$, we obtain,
for $s_1,s_2\in \JO$, the interior of $J$,
\begin{equation}\label{mad}
  \frac{\partial^2}{\partial s_1\partial s_2}
\E\bigpar{M_t(s_1)M_t(s_2)}
=
  \frac{\partial^2}{\partial s_1\partial s_2} h(s_1,s_2;t),
\end{equation}
where the \rhs{} exists because 
\eqref{boss} holds and,
as said above,
$\gf$ is
twice continuously
differentiable; furthermore, this implies, using the explicit form of
$h(s_1,s_2;t)$ in \eqref{OFF},
\begin{equation}\label{mah}
  \begin{split}
  \frac{\partial^2}{\partial s_1\partial s_2} h(s_1,s_2;t)
&= O \Bigpar{1+(1+t^2) e^{t\Re(\gf(s_1+s_2)-\gf(s_1)-\gf(s_2))}}
\\&
= O \Bigpar{1+(1+t^2) e^{-t/2}}
=O(1),
\qquad s_1,s_2\in \JO.
  \end{split}
\end{equation}
In the \lhs{} of \eqref{mad} we interchange the order of differentiation
and expectation. To justify this, we note first that
by \eqref{F}, $|F_t(s)|\le |\ctt|$ and
$|\dds F_t(s)|\le\sum_{v\in \ctt} |X_v|$. Hence, using \eqref{M} and 
Lemmas \ref{L3a} and \ref{L3b}, and
letting $C_t$ denote constants that may depend on $t$ but not on $s_1,s_2$,
\begin{equation}
  \begin{split}
  &  \E \Bigabs{ \frac{\partial^2}{\partial s_1\partial s_2}
\bigpar{M_t(s_1)M_t(s_2)}}
=   \E \Bigabs{ \frac{\partial M_t(s_1)}{\partial s_1}
\frac{\partial M_t(s_2)}{\partial s_2}}
\\&\qquad
\le
C_t\E\Bigpar{\Bigpar{|F_t(s_1)|+\Bigabs{\dds F_t(s_1)}}
\Bigpar{|F_t(s_2)|+\Bigabs{\dds F_t(s_2)}}
  }
\\&\qquad
\le  C_t \E \Bigpar{ |\ctt| +\sum_{v\in \ctt} |X_v|}^2
\le C_t\E|\ctt|^2+  C_t \E \Bigpar{ \sum_{v\in \ctt} |X_v|}^2
\le
C_t<\infty.
  \end{split}
\end{equation}
Consequently, if $[a,b]$ and $[c,d]$ are any two subintervals of $\JO$, then
Fubini's theorem yields
\begin{align}
    \int_a^b\int_c^d  &
\E\Bigpar{ \frac{\partial M_t(s_1)}{\partial s_1}
\frac{\partial M_t(s_2)}{\partial s_2}
}
\dd s_1\dd s_2
=
  \E 
\int_a^b\int_c^d  
{ \frac{\partial M_t(s_1)}{\partial s_1}
\frac{\partial M_t(s_2)}{\partial s_2}}
\dd s_1\dd s_2
\notag\\&
=\E\bigpar{\bigpar{M_t(b)-M_t(a)}\bigpar{M_t(d)-M_t(c)}}
\notag\\&
=h(b,d;t)-h(a,d;t)-h(b,c;t)+h(a,c;t)
\end{align}
and differentiation (with respect to $b$ and $d$) yields 
the desired formula
\begin{align}\label{mak}
\E\Bigpar{ \frac{\partial M_t(s_1)}{\partial s_1}
\frac{\partial M_t(s_2)}{\partial s_2}
}
=
  \frac{\partial^2h(s_1,s_2;t)}{\partial s_1\partial s_2} ,
\qquad s_1,s_2\in \JO.
\end{align}
Together with \eqref{mah}, this shows, for $s\in \JO$,
\begin{align}
\E\Bigabs{ \frac{\partial M_t(s)}{\partial s}}^2
&=
-\E\Bigpar{ \frac{\partial M_t}{\partial s}(s)
\frac{\partial M_t}{\partial s}(-s)
}
=
-  \frac{\partial^2h(s_1,s_2;t)}{\partial s_1\partial  s_2}\Bigm|_{s_1=s,s_2=-s} 
\notag\\&
=O(1).
\label{mal}
\end{align}
Finally, we use \eqref{mal} together with \eqref{lm} in \eqref{emma}, and
find
\begin{equation}\label{cup}
  \E \normWWJ{M_t}^2 \le C,
\qquad t\ge0.
\end{equation}
In other words, $\set{M_t}$ is an $L^2$-bounded family of random variables
in $\WWJ$. 

In particular, each $M_t$ is integrable, and thus the conditional
expectation
$\E(M_t\mid \cF_u)$ is defined for every $u\le t$.
Point evaluations are continuous linear functionals on $\WWJ$ by
\eqref{kg}. Hence, if $0\le u\le t$ and $s\in J$,
then, using also that $M_t(s)$ is a martingale by \refL{LM},
\begin{equation}
\E(M_t\mid \cF_u)(s)
= \E(M_t(s)\mid \cF_u)
=M_u(s).
\end{equation}
Consequently, $\E(M_t\mid \cF_u)=M_u$, and thus $M_t$, $t\ge0$, is a
martingale with values in $\WWJ$.
We have shown $L^2$-boundedness in \eqref{cup}.
Finally, $t\mapsto M_t$ is right-continuous by the definition \eqref{M},
since $F_t$ is a right-continuous step function.
\end{proof}

\begin{lemma}
  \label{LC}
There exists a random function $M_\infty\in\WWJ\subset C(J)$ such that
$M_t\asto M_\infty$ in $\WWJ$ as \ttoo.
\end{lemma}
\begin{proof}
  This follows from \refL{LMWWJ} since $\WWJ$ is a Hilbert space and thus
  has the \RNP, see \cite[Theorem 2.9 and Corollary 2.15]{Pisier}, using 
\cite[Theorem 1.49]{Pisier} to extend the result from the discrete-parameter
martingale $(M_n)_{n\in\bbN}$ to the continuous-parameter $(M_t)_{t\ge0}$.
\end{proof}

\begin{lemma}
  \label{LCC}
As \ttoo, $M_t\asto M_\infty$ in $C(J)$, \ie, \as{} $M_t(s)\to M_\infty(s)$
uniformly in $s\in J$.
\end{lemma}
\begin{proof}
An immediate consequence of \refL{LC} and \eqref{kg}.
\end{proof}

\begin{lemma}\label{L0}
  $M_\infty(0)>0$ a.s.
\end{lemma}
\begin{proof}
  By \eqref{M} and \eqref{F},
  \begin{equation}
M_t(0)=e^{-t}F_t(0)= e^{-t}|\ctt|, 
  \end{equation}
which \as{} converges to the strictly positive limit $W$ by \eqref{winston}.
\end{proof}

\begin{remark}
  \label{Rrho}
We have so far assumed that $\rho=1$.  
The results easily extend to 
general $\rho>0$, provided we, as in Sections \ref{Stree1}--\ref{Sbranch},
count the nodes according to their weights $\ww_v$ and
thus change \eqref{F} to
\begin{equation}
  \label{Frho}
F_t(s):=\sum_{v\in\ctt}\ww_ve^{\ii s X_v}=\rho+\sum_{v\neq o}e^{\ii sX_v}.
\end{equation}
Recall that $|\ctt|$ now is the total weight; thus $F_t(0)=|\ctt|$ still holds. 
Similarly, recalling \eqref{bmuw}, the \chf{} of $\bmu_t$ is still
$F_t(s)/F_t(0)$.
In the proof of \refL{L1}, we then obtain the same differential
equations  \eqref{er66} and \eqref{erII}, 
but the initial condition is now $F_0(s)=\rho$, giving
\begin{align}\label{jb}
  \E F_t(s)&=\rho e^{t\gf(s)}
\end{align}
and
\begin{align}\label{RFF}
      \E \bigpar{F_t(s_1)F_t(s_2)} 
= 
\rho^2 e^{t(\gf(s_1)+\gf(s_2))}
&+\rho
\frac{\gf(s_1+s_2)}{\gf(s_1)+\gf(s_2)-\gf(s_1+s_2)}
\times
\notag\\&\quad
\bigpar{e^{t(\gf(s_1)+\gf(s_2))}-e^{t\gf(s_1+s_2)}}
.
     \end{align}
We define $M_t(s):=F_t(s)/\E F_t(s)=\rho\qw e^{-t\gf(s)}F_t(s)$
and obtain again the estimate \eqref{lm}. The rest of the proofs above
holds without changes; in particular, Lemmas \ref{LMWWJ}--\ref{L0} hold as
stated for any $\rho>0$.
\end{remark}

\begin{proof}[Proof of \refT{TYule}]
We first continue to assume $d=1$, but allow $\rho>0$ to be arbitrary, see
\refR{Rrho}. 

Let $\bmu_t$ be the empirical distribution \eqref{bmu} for the Yule tree
$\ctt$ at time $t\ge0$.
Denote the \chf{} of $\bmu_t$ by $\hbmu_t$. Then, by \eqref{bmu},
\eqref{F} and \eqref{M} when $\rho=1$, 
and their modifications \eqref{bmuw} and \eqref{Frho} in general,
using $\gf(0)=1$,
\begin{equation}\label{fle}
  \hbmu_t(s)=
\frac{1}{|\ctt|}\sum_{v\in \ctt}\ww_v e^{\ii s X_v}
= \frac{F_t(s)}{F_t(0)}
= \frac{M_t(s)}{M_t(0)}e^{t(\gf(s)-1)}.
\end{equation}
If $s(t)$ is a function of $t$ such that
$s(t)\to0$ as \ttoo, then the uniform convergence in \refL{LCC} together
with the continuity of $M_\infty$ implies that 
$M_t(s(t))\asto M_\infty(0)$, and thus, using also \refL{L0},
$M_t(s(t))/M_t(0)\asto M_\infty(0)/M_\infty(0)=1$.
Consequently, \eqref{fle} shows that \as{}, as \ttoo,
\begin{equation}\label{flex}
  \hbmu_t(s(t))=\bigpar{1+o(1)}e^{t(\gf(s(t))-1)}
=e^{t(\gf(s(t))-1)+o(1)}.
\end{equation}
Now consider the rescaled measure $\Thetaab(\bmu_t)$, where we take 
$a=a(t):=\sqrt t$
and $b=b(t):=\bm t$. 
By \eqref{Thetabmu}, its \chf{} is given by
\begin{equation}\label{kai}
  \widehat{\Thetaab(\bmu_t)}(s)
=\frac{1}{|\ctt|}\sum_{v\in\ctt}e^{\ii s(X_v-b)/a}
=e^{-\ii sb/a}\hbmu_t(s/a).
\end{equation}
For any fixed $s\in\bbR$, $s/a=s/\tqq\to0$ as \ttoo, and thus \eqref{flex}
applies with $s(t):=s/a=s/\tqq$; hence \eqref{kai} yields, \as,
\begin{equation}\label{kais}
  \widehat{\Thetaab(\bmu_t)}(s)
=e^{-\ii sb/a}
 e^{t(\gf(s/a)-1)+o(1)}
=e^{t(\gf(s/\tqq)-1-\ii  \bm s/\tqq)+o(1)}.
\end{equation}
Furthermore, $\gf(s/\tqq)=1+\ii\bm s/\tqq-\frac12 \E[\eta^2]s^2/t+o(t\qw)$,
and thus a.s., as \ttoo,
\begin{equation}\label{kaisa}
  \widehat{\Thetaab(\bmu_t)}(s)
=e^{-\frac12\E[\eta^2]s^2+o(1)}
\to e^{-\frac12\E[\eta^2]s^2}.
\end{equation}
This shows that for each fixed $s$, the characteristic function of
$\Thetaab(\bmu_t)$ converges \as{} to the \chf{} of $N(0,\bm^2+\gss)$.

Consider now an arbitrary $d\ge1$.
For any fixed $u\in\bbR^d$, consider the linear projections $u\tr X_v\in\bbR$,
which are obtained as in \eqref{X} from the offsets $u\tr\eta_v$.
Applying \eqref{kaisa} (with $s=1$) to these variables yields
\begin{equation}\label{kaisb}
  \widehat{\Thetaab(\bmu_t)}(u)
\asto e^{-\frac12\E[(u\tr\eta)^2]}
= e^{-\frac12 u\tr\E[\eta\eta\tr]u}
\end{equation}
for every fixed $u\in\bbR^d$.
This implies that $\Thetaab(\bmu_t)\dto N(\bigpar{0,\E[\eta\eta\tr]}$ \as,
see \cite{BertiEtAl}.
\end{proof}

\begin{remark}\label{Rpoisson}
  The final part of the proof, from \eqref{flex}, is very similar to
  standard proofs of the central limit theorem, and we can interpret
  \eqref{flex} as showing that $\bmu_t$ asymptotically is like the
  distribution of a sum of independent copies of $\eta$; note that
  \eqref{flex} says that $\hbmu_t(s)\approx e^{t(\gf(s)-1)}$, which is the
  \chf{} of a sum of a random $\Po(t)$ number of independent copies of $\eta$.
\end{remark}

\begin{remark}
  \label{Rspaces}
The main idea in the proof above, as in many other related works, including
\cite{MM},  is to obtain 
uniform convergence of certain random functions in some interval $J\ni0$,
\ie, convergence in $C(J)$,
since this allows us to obtain convergence for an argument $s=s(t)$
depending on $t$. (\refL{LCC}.)
Pointwise convergence \as{} to a random function follows in our cases, and
in many related problems, from the martingale limit theorem.
One method to improve this to uniform convergence 
goes back 
to \citet{JoffeLecamNeveu};
the idea is to use the Kolmogorov continuity criterion 
\cite[Theorem 3.23]{Kallenberg}
to show that the limit can be taken as a continuous function;
then uniform convergence follows by the martingale convergence theorem in
a space of continuous functions.
We have here chosen a slightly different method; 
we use a Sobolev space $\WWJ$ and show that the martingale is bounded there;
then both existence of the limit and uniform convergence follows.
Nevertheless the methods are quite similar; the first requires estimates of
moments of differences while the second requires estimates of moments of
derivatives, and the required estimates are similar. Hence, for practical
applications, the two methods seem to be essentially equivalent.

\citet{Uchiyama} uses pointwise convergence of random functions; he does not
explicitly show uniform convergence, but he too uses estimates on moments of
differences, in a way which  seems related. 

Another method to obtain uniform continuity, see
\citet{Biggins-uniform1d,Biggins-uniform},
assumes that the random
functions are analytic functions in an open domain in the complex plane.  
Then pointwise estimates of moments yield automatically (by
Cauchy's estimates)
uniform estimates  of the functions and their derivatives on compact sets,
and thus uniform convergence on compact subsets.
(It may be convenient to use the Bergman space of square
integrable analytic functions in a suitable domain,
\cf{}  \cite{SJ194}.) 
This is the method used by \citet{MM} for the \Polya{} urns and \rrt{s}
discussed in the present paper; it is elegant but it requires in our case
exponential moments of the offset distribution so that its
 \chf{} can be extended to an
analytic functions in a complex domain.

\citet{ChauvinDJ} (for binary trees) use a combination of both the
Kolmogorov criterion
and properties of analytic functions.

Note that also when the offset is vector-valued and takes values in $\bbR^d$
with $d>1$, we consider one-dimensional projections and use a
one-dimensional Sobolev space.
We may define $M_t(\vecs)$ as above for $\vecs\in\bbR^d$ and show, by the
same arguments, that
$M_t(\vecs)$ is a $L^2$-bounded martingale in the Sobolev space $\WW(B)$,
for a small ball $B\subset\bbR^d$. However, we cannot use this to claim
convergence  in $C(B)$ (or in a smaller ball), 
since the Sobolev imbedding theorem in higher
dimensions require more derivatives, see 
\cite[Theorem 5.4]{Adams}.
(One might use a Sobolev space $\WWk(B)$
with more derivatives, but that would require more moments for the offset
distribution, apart from complicating the proof.) 
\end{remark}

\section{Proof of \refT{Trrt}}\label{SpfTrrt}

\begin{proof}[Proof of \refT{Trrt}]
Since the (weighted) \rrt{} can be realized as the (weighted) Yule tree at 
the stopping times $\tau_n$ defined by \eqref{taun}, we obtain by
taking $t=\tau_n$ in \eqref{tyule}, where $\bmu_n$ now refers to the \rrt,
\begin{equation}\label{tyulerrt}
  \Theta_{\sqrt{\tau_n},\,\bm \tau_n}(\bmu_n) \asto
N\bigpar{0,\E[\eta\eta\tr]}
.
\end{equation}
Combined with \eqref{tauas}, this implies \eqref{trrt1}.
To see this in detail, we can write \eqref{tyulerrt} in the form \eqref{trrt3}:
conditioned on \set{\ctt} and \set{\eta_v}, a.s,
\begin{equation}\label{trrt3x}
  \frac{X_{V_n}-\bm\log \tau_n}{\sqrt{\log \tau_n}}
\dto 
N\bigpar{0,\E[\eta\eta\tr]}.
\end{equation}
This implies \eqref{trrt3} by \eqref{tauas} and the Cram\'er--Slutsky theorem 
(still conditioning on \set{\ctt} and \set{\eta_v}).
\end{proof}

We have chosen to prove \refT{Trrt} using the continuous-time Yule tree.
However, it is also possible to argue directly in discrete time in the same
way.
We find it  interesting to  sketch this version of the argument too, 
for comparison, leaving some details to the reader;
see also the proof of
\cite[Theorem 1.6]{MM}, where the result is proved under stronger
assumptions using similar and partly the same arguments.
We consider an arbitrary $\rho>0$.

Define $F_n(s)$ as in \eqref{Frho}. Then, labelling the nodes in order of
appearance, 
$F_{n+1}(s)=F_n(s)+e^{\ii s X_{n+1}}$ and thus,
\cf{} \eqref{stoc},
\begin{align}\label{jc}
 \E\bigpar{F_{n+1}(s)\mid\cF_n}
&=F_n(s) + \frac{1}{n+\rho} \sum_{v\in T_{n} } \go_v e^{\ii s X_v}\gf(s)
=F_n(s) + \frac{\gf(s)}{n+\rho} F_n(s)
\notag\\&
= \frac{n+\rho+\gf(s)}{n+\rho}F_n(s).
\end{align}
In particular,
\begin{equation}
  \E {F_{n+1}(s)}
= \frac{n+\rho+\gf(s)}{n+\rho}\E F_n(s)
\end{equation}
and thus by induction, since $F_0(s)=\rho$,
\begin{equation}\label{cq}
  \E {F_{n}(s)}
= \frac{\gG(\rho+1)}{\gG(\rho+\gf(s))}\frac{\gG(n+\rho+\gf(s))}{\gG(n+\rho)}.
\end{equation}
While this formula is more complicated than \eqref{jb}, it is still easy to
get asymptotics. 
As a well-known consequence of Stirling's formula, see 
\cite[(5.11.12)]{NIST},
for any complex $a,b$,
\begin{equation}\label{nisse}
  \frac{\gG(n+a)}{\gG(n+b)} =\bigpar{1+o(1)}n^{a-b},
\qquad \ntoo.
\end{equation}
Hence, 
\begin{equation}\label{cqu}
  \E {F_{n}(s)} 
=\bigpar{1+o(1)} \frac{\gG(\rho+1)}{\gG(\rho+\gf(s))} n^{\gf(s)},
\qquad \ntoo.
\end{equation}

For the second moment, we similarly obtain, \cf{} \eqref{erII},
\begin{align}
&
\begin{aligned}
  \E\bigpar{F_{n+1}(s_1)F_{n+1}(s_2)\mid\cF_n}
&=F_n(s_1)F_n(s_2)+\frac{1}{n+\rho}\Bigl(
\gf(s_1)F_n(s_1)F_n(s_2)
\\&\quad{}+\gf(s_2)F_n(s_1)F_n(s_2)
+\gf(s_1+s_2)F_n(s_1+s_2)
\Bigr)
\end{aligned}
\notag\\&\quad
=
\frac{n+\rho+\gf(s_1)+\gf(s_2)}{n+\rho}F_n(s_1)F_n(s_2)
+\frac{\gf(s_1+s_2)}{n+\rho}F_n(s_1+s_2)
\end{align}
and thus, using induction and \eqref{cq},
\begin{multline}
  \E\bigpar{F_{n}(s_1)F_{n}(s_2)}  
=
\frac{\gG(n+\rho+\gf(s_1)+\gf(s_2)}{\gG(n+\rho)}
\biggl(\frac{\rho^2\gG(\rho)}{\gG(\rho+\gf(s_1)+\gf(s_2))}
\\
+\sum_{k=1}^n 
\frac{\gG(k+\rho+\gf(s_1+s_2)-1)}{\gG(k+\rho+\gf(s_1)+\gf(s_2))}
\frac{\gG(\rho+1)\gf(s_1+s_2)}{\gG(\rho+\gf(s_1+s_2))}
\biggr)
.
\label{effy}
\end{multline}

We define, similarly to \eqref{M},
$M_n(s):=F_n(s)/\E F_n(s)$
(at least for $s\in J$, which implies $F_n(s)\neq0$ by \eqref{cq}).
It follows from \eqref{jc} that $M_n(s)$ is a martingale for each $s$.
Furthermore, it follows from \eqref{effy}, \eqref{nisse} and \eqref{cqu} that,
uniformly for $s_1,s_2\in J$, recalling \eqref{boss},
\begin{equation}
  \E\bigpar{M_n(s_1)M_n(s_2)}
=O\Bigpar{1+\sumkn k^{-3/2}}
=O(1).
\label{mega}
\end{equation}
Moreover, \eqref{nisse} holds uniformly for $a$ and $b$ in any fixed
bounded sets in $\bbC$, and thus (using Cauchy's estimates), we can
differentiate \eqref{nisse} with respect to $a$ and $b$, arbitrarily many times.
It thus follows from \eqref{effy} and \eqref{cq} that also, for $s_1,s_2\in\JO$,
\begin{equation}
  \E\Bigpar{\frac{\partial M_n(s_1)}{\partial s_1}
\frac{\partial M_n(s_2)}{\partial s_2}}
=O\Bigpar{1+\sumkn \frac{\log^2 k}{k^{3/2}}}
=O(1).
\label{dux}
\end{equation}
Hence, \cf{} \eqref{mak}--\eqref{cup},
$M_n$ is an $L^2$-bounded martingale in $\WWJ$.
We can now argue as in \refS{SpfTYule}, and obtain for any $s(n)\to0$, \cf{}
\eqref{flex},
\begin{equation}\label{plex}
  \hbmu_n(s(n))=\bigpar{1+o(1)}n^{\gf(s(n))-1}
=e^{\log n(\gf(s(n))-1)+o(1)},
\end{equation}
and the proof is completed as in \refS{SpfTYule}.
(A minor simplification is that now $F_n(0)=n+\rho$ and thus $M_n(0)=1$
deterministically, so \refL{L0} is not needed.)

We thus see that there is no essential difference between using the discrete
time 
\rrt{} or the continuous time Yule tree in our arguments. However,
the formulas for the first and, in particular, second moments are much
simpler in the continuous time case, which leads to simpler calculations in
order to obtain the desired estimates.

\section{Proof of \refT{TP}}\label{SpfTP}

We extend the notations \eqref{FT} and \eqref{gf} to $\bbR^d$.
Let $a_n:=\sqrt{\log n}$ and $b_n:=m\log n$.

\begin{proof}[Proof of \refT{TP}]
\stepx \label{s1}
Consider first  an \SRW{} \Polya{} urn, with initial composition
$\rho\gd_0$, for some $\rho>0$.
Then, as shown in \refSS{SSkorr1}, the sequence $\mu_n$ is the same as for a
weighted \rrt, and thus \eqref{tp} follows from \refT{Trrt}.

\stepx\label{s2}
Consider now a \DRWPU{} $\mu_n$ with initial
composition $\mu_0=\rho\nu=\rho\gd_0*\nu$ for some $\rho\ge0$. 
Then \refL{Lux} and \eqref{*} show that the urn is given by
$\mu_n=\mu'_n*\nu$ for an \SRW{} Polya{} urn $\mu_n'$, with the same
offset distribution $\nu$ and initial composition $\mu'_0=\rho\gd_0$.
Hence, $\tmu_n=\tmu'_n*\nu$, and thus, 
\cf{} \eqref{kai}, for every $s\in\bbR^d$,
\begin{equation}\label{qai}
  \begin{split}
  \widehat{\Thetaanbn(\bmu_n)}(s)
&=e^{-\ii s\cdot b_n/a_n}\hbmu_n\san
=e^{-\ii s\cdot b_n/a_n}\widehat{\bmu'_n}\san\widehat\nu\san
\\&
=  \widehat{\Thetaanbn(\bmu'_n)}(s)\,\widehat\nu\san.    
  \end{split}
\end{equation}
Here, $\widehat\nu\xsan\to1$  as \ntoo{} since $ a_n\to\infty$,
and Step \ref{s1} shows that \as{}, for all $s$,
\begin{equation}\label{qaisa}
\widehat{\Thetaanbn(\bmu'_n)}(s)
\to e^{-\frac12\E[\eta^2]s^2}.
\end{equation}
Hence, 
\eqref{qai} yields
\as, 
for all $s\in\bbR^d$,
\begin{equation}\label{qaisb}
\widehat{\Thetaanbn(\bmu_n)}(s)
\to e^{-\frac12\E[\eta^2]s^2},
\end{equation}
which proves \eqref{tp} in this case.

\stepx\label{s3}
Consider now a \DRWPU{} $\mu_n$, with 
arbitrary initial composition $\mu_0\in\cMx(\bbR^d)$.
Let $\rho:=\mu_0(\bbR^d)$, and consider a \Polya{} urn with the same offset
distribution $\nu$ but initial composition $\mu'_0:=\rho\nu$.
Then Step \ref{s2} applies to the second urn, which shows
that
\as, 
for all $s\in\bbR^d$,
\begin{equation}\label{qais3a}
\widehat{\Thetaanbn(\bmu'_n)}(s)
\to e^{-\frac12\E[\eta^2]s^2}.
\end{equation}
By \refL{LP} below, this implies
\begin{equation}\label{qais3b}
\widehat{\Thetaanbn(\bmu_n)}(s)
\asto e^{-\frac12\E[\eta^2]s^2},
\end{equation}
for every $s\in\bbR^d$, which implies \eqref{tp}.
(We may again use \cite{BertiEtAl}, or note that the proof of \refL{LP} shows
that \as{} \eqref{mans} holds for all $s$ simultaneously.)

\stepx\label{s4}
Finally, consider an \SRW{} \Polya{} urn $\mu_n'$  with arbitrary
initial composition.
We use \refL{Lux} and construct a corresponding \DRW{} \Polya{} urn
$\mu_n=\mu_n'*\nu$. Step \ref{s3} applies to the latter urn, 
which shows \eqref{qaisb} \as, and then \eqref{qai} shows \eqref{qaisa}
which proves the result in this case.

This completes the proof of \refT{TP}, assuming the lemma below.
\end{proof}

\begin{lemma}\label{LP}
  Let $\mu_n$ and $\mu'_n$ be two \DRW{} \Polya{} urn processes, with 
the same offset distribution $\nu\in\cP(\bbR^d)$ but
(possibly) different initial distributions $\mu_0$ and
  $\mu'_0$. Assume that $\mu_0(\bbR^d)=\mu_0'(\bbR^d)$.
Then, for any sequence $s_n\in\bbR^d$ with $s_n\to0$,
as \ntoo,
\begin{align}\label{maans}
\bigabs{\widehat{\tmu_n}(s_n)-\widehat{\tmu_n'}(s_n)}
\asto0.
\end{align}
Hence, for every $s\in \bbR^d$,
\begin{align}\label{mans}
\bigabs{\widehat{\Thetaanbn(\tmu_n)}(s)-\widehat{\Thetaanbn(\tmu_n')}(s)}
\asto0.
\end{align}
\end{lemma}
\begin{proof}
  We use the coupling with a weighted \rrt{} $T_n$ in \refSS{SSkorrD},
and may assume that both \Polya{} urns are defined by \ref{PU0}--\ref{PU}
from the same trees $T_n$. We use the notations $\gD\mu'_v$ and $Z_v'$ for
the second urn. Let $\ZZ_v:=Z_v'-Z_v\in\bbR^d$.

Recall that by \eqref{gdmu}, for any $v\neq o$ and conditioned on $Z_v$, 
we have
\begin{equation}\label{gd1}
\gD\mu_v 
=\cL(Z_v+\eta).  
\end{equation}
Hence, by \ref{PU}, if $v$ has a mother $u\neq o$, then, conditioned on
the tree process $(T_n)$ and on
$Z_u$ and $Z_u'$,
$Z_v$ has the distribution $\cL(\eta+Z_u)$, i.e., 
$ 
Z_v \eqd \eta+Z_u,
$ 
and
\begin{equation}
Z'_v \eqd \eta+Z'_u
= \eta+ Z_u+\ZZ_u
\eqd Z_v+\ZZ_u.
\end{equation}
Hence, we may couple the two urn processes such that when $v$ has a mother
$u\neq o$, then
\begin{equation}\label{mamma}
  \ZZ_v:=Z'_v-Z_v = \ZZ_u.
\end{equation}
If $v$ is a daughter of $o$, we may, for example, choose $Z_v$ and $Z_v'$
independent. 

Note also that \eqref{gdmu} implies, 
\begin{equation}
\widehat{\gD\mu_v}(s)= \E \bigpar{e^{\ii s\cdot (Z_v+\eta)}\mid Z_v} 
=e^{\ii s \cdot Z_v} \E {e^{\ii s\cdot\eta}}
=e^{\ii s\cdot Z_v} \gf(s).
\end{equation}

For any node $v\neq o$, let $u(v)$ be the ancestor of $v$ that is a daughter
of $o$. (With $u(v)=v$ if $v$ is a daughter of $o$.)
Then, repeated application of \eqref{mamma} shows that $\ZZ_v=\ZZ_{u(v)}$.
Consequently, for $v\neq o$,
\begin{multline}\label{anna}
\bigabs{\widehat{\gD\mu_v}(s)  - \widehat{\gD\mu'_v}(s) }
= \bigabs{e^{\ii s\cdot Z_v} \gf(s) - e^{\ii s\cdot Z_v'} \gf(s)}
= \bigabs{e^{\ii s\cdot Z_v}- e^{\ii s\cdot Z_v'}}| \gf(s)|
\\
\le  \bigabs{e^{\ii s\cdot Z_v}- e^{\ii s\cdot Z_v'}}
= \bigabs{e^{\ii s\cdot \ZZ_v}- 1}
= \bigabs{e^{\ii s\cdot \ZZ_{u(v)}}- 1}.  
\end{multline}

We sum \eqref{anna} over all $v\in T_n$ with $v\neq o$. 
Let $D$ be the set of nodes in $T_\infty$ that are daughters of $o$, and for
each $u\in D$, let $N_n(u)$ be the number of descendants of $u$ in $T_n$.
Then, with $\rho=\mu_0(\bbR)=\mu'_0(\bbR)$,
\begin{align}
\bigabs{\widehat{\mu_n}(s)-\widehat{\mu_n'}(s)}
&
\le \sum_{v\in T_n} \bigabs{\widehat{\gD\mu_v}(s)- \widehat{\gD\mu'_v}(s)}
\notag\\&
\le \bigabs{\widehat{\gD\mu_0}(s)  - \widehat{\gD\mu'_0}(s) }
+ \sum_{o\neq v\in T_n} \bigabs{e^{\ii s\cdot \ZZ_{u(v)}}- 1}
\notag\\&
\le 2\rho+ \sum_{u\in D} N_n(u) \bigabs{e^{\ii s\cdot \ZZ_{u}}- 1}.
\end{align}
Dividing by $|T_n|=n+\rho$ we find for the normalized distributions
\begin{align}\label{bull}
\bigabs{\widehat{\tmu_n}(s)-\widehat{\tmu_n'}(s)}
\le \frac{2\rho}{n}+ \sum_{u\in D} \frac{N_n(u)}{n} 
\bigabs{e^{\ii s\cdot \ZZ_{u}}- 1}.
\end{align}
Number the elements of $D$ as $u_1,u_2,\dots$ (in order of appearance).
It is well known, at least  in the unweighted case, that the relative sizes
$N_n(u_k)/n$ of the branches converge \as{} as \ntoo, say
\begin{equation}\label{chinese}
  N_n(u_k)/n\asto V_k, \qquad k\ge 1.
\end{equation}
In fact, this is an instance of the Chinese restaurant process with 
parameters (seating plan) $(0,\rho)$, see \cite[Section 3.2]{Pitman};
take one table for each  $u\in D$ and all its descendants, 
and then the probability
that node $n+1$ joins a table with $n_k$ nodes is $n_k/(n+\rho)$, and the
probability that it starts a new table (and thus belongs to $D$) is
$\rho/(n+\rho)$. 
Hence, \eqref{chinese} follows from \cite[Theorem 3.2]{Pitman}, which
furthermore shows that $V_k=W_k\prod_{j=1}^{k-1}(1-W_j)$
for an \iid{} sequence $W_j\sim \operatorname{Beta}(1,\rho)$.
(This result can also be proved easily using a sequence of classical
two-colour  \Polya{} urns, \cf{} \eg{} \cite[Appendix A]{SJ320}.)
It follows that, \as,
\begin{equation}\label{kul}
  \sumk V_k = 1-\prod_{j=1}^\infty (1-W_j) = 1,
\end{equation}
and thus the limits $V_k$ in \eqref{chinese} form a random probability
distribution on $\bbN$. (The random distribution $\operatorname{GEM}(0,\rho)$
\cite{Pitman}.) For a discrete distributions, pointwise convergence of the
probabilities to a limiting distribution is equivalent to convergence in
total variation, see \cite[Theorem 5.6.4]{Gut}, and thus \eqref{chinese}
and \eqref{kul} imply
\begin{equation}\label{kuli}
\sumk  \lrabs{\frac{ N_n(u_k)}{n}- V_k} \asto 0.
\end{equation}
We now return to \eqref{bull} and find, for any sequence $s_n\in\bbR^d$,
\begin{align}\label{bill}
\lrabs{\widehat{\tmu_n}(s_n)-\widehat{\tmu_n'}(s_n)}
\le \frac{2\rho}{n}
+ \sumk V_k \lrabs{e^{\ii s_n\cdot \ZZ_{u_k}}- 1}
+ \sumk 2\lrabs{\frac{ N_n(u_k)}{n}- V_k}.
\end{align}
Here the last sum converges to 0 \as{} by \eqref{kuli}.
Furthermore, note that $\ZZ_{u_k}$ does not depend on $n$.
Hence, if $s_n\to0$, then the first sum converges to 0 \as, by 
\eqref{kul} and dominated convergence. 
Consequently, \eqref{maans} follows from \eqref{bill}.
Finally, \eqref{mans} follows because
$  \widehat{\Thetaanbn(\bmu_n)}(s)
=e^{-\ii s\cdot b_n/a_n}\hbmu_n\xsan$
and similarly for $\bmu_n'$.
\end{proof}

\section{Binary trees}
\label{Sbinary}

Theorems \ref{Trrt}--\ref{TYule} have analogues for \bst{s} and binary Yule
trees.
(In fact,  the unweighted cases of \refTs{Trrt}--\ref{TYule} can be seen as
special cases of the results for binary trees, see \refE{Ebstrrt} below.)
We let in the present section  $T_n$ and $\ctt$ denote these random binary
trees.
Recall from \refS{Stree1} that in the these trees, nodes are either internal
(dead) or external (living), and that each internal node has two children,
labelled left and right. If $v$ is an internal node, we denote its left and
right child by $\vl$ and $\vr$, respectively.

In this context, it is natural to generalize the definitions in
\refS{Sbranch}, and allow the offsets to have different distributions for
left and right children; furthermore, we may allow a dependency between the
offsets of the two children of a node. 
Hence, in this section we assume,
instead of \ref{BWeta} in \refS{Sbranch}
(see also \refR{R1}),
that we are given 
a random variable $\etax=(\eta\subL,\eta\subR)\in\bbR^d\times\bbR^d$
with distribution $\nux:=\cL(\etax)$,
and that
$\etax_v=(\eta_\vl,\eta_\vr)$, $v\in T$, are
\iid{} copies of $\etax$. This defines $\eta_w$ for every $w\neq o$, which
is enough to define $X_v$ as before; recall \eqref{X} and \refR{R2}.

This setting (for the binary Yule tree)
is a special case of the one in \citet{Uchiyama}, where also, more
generally, the number of children may vary.
\refT{TYuleb}  below (at least the external case)
is thus a special case of \cite[Theorem 4]{Uchiyama},
(and \refT{Tbst} an easy consequence), but as in
\refSs{SpfTYule}--\ref{SpfTrrt} we give (rather) complete proofs with
explicit calculations for both binary Yule trees and \bst{s},
for comparison with the proofs of \refTs{Trrt}--\ref{TYule} and other
similar proofs in the literature.

This setting has also been used by
\citet{Fekete}, who showed (among other results) a weaker
version of \refT{Tbst} below with convergence
in probability; we improve this to convergence a.s.

Since we have two types of nodes, we define besides $\bmu$ given by
\eqref{bmu}, where we use all nodes, also the internal and external
versions.
Let $\Ti$ and $\Te$ be the sets of internal and external nodes in $T$,
respectively, and let
\begin{align}
  \bmui&:=\frac{1}{|\Ti|}\sum_{v\in\Ti}\gd_{X_v},
&
  \bmue&:=\frac{1}{|\Te|}\sum_{v\in\Te}\gd_{X_v}.
\end{align}
Note that, since the trees are binary, $|\Te|=|\Ti|+1$, and thus
$|T|=|\Ti|+|\Te|=2|\Ti|+1$. 
In particular, for the \bst, 
$|\Ti_n|=n$,
$|\Te_n|=n+1$
and $|T_n|=2n+1$.

Let
\begin{equation}
  \label{mb}
  m:=\E \eta\subL+\E\eta\subR\in\bbR^d.
\end{equation}
\begin{theorem}
  \label{Tbst}
Let\/ $T_n$ be the \bst{} and
suppose that\/ $\E|\etax|^2<\infty$.
Then, as \ntoo,
in $\cP(\bbR^d)$,
\begin{equation}\label{tbst1}
  \Theta_{\sqrt{\log n},\,\bm\log n}(\bmu_n) \asto
N\bigpar{0,\E[\etal\etal\tr]+\E[\etar\etar\tr]}
.
\end{equation}
Furthermore, the same result holds for the internal and external empirical
distributions $\bmui_n$ and $\bmue_n$.
\end{theorem}

\begin{theorem}[\citet{Uchiyama}]
  \label{TYuleb}
Let\/ $\ctt$ be the binary Yule tree and
suppose that\/ $\E|\etax|^2<\infty$.
Then, as \ttoo,
in $\cP(\bbR^d)$,
\begin{equation}\label{tyuleb}
  \Theta_{\sqrt{t},\,\bm t}(\bmu_t) \asto
N\bigpar{0,\E[\etal\etal\tr]+\E[\etar\etar\tr]}
.
\end{equation}
Furthermore, the same result holds for the internal and external empirical
distributions $\bmui_t$ and $\bmue_t$.
\end{theorem}

\begin{remark}\label{Rbins}
  Note that in Theorems \ref{Tbst}--\ref{TYuleb} it does not matter whether
  there is a dependency between $\etal$ and $\etar$ or not; $m$ and the
  asymptotic variance $\E[\etal\etal\tr]+\E[\etar\etar\tr]$ are obtained by
  summing separate contributions from $\etal$ and $\etar$. This is not
  surprising, since if we consider $X_v$ for a fixed node $v$, it is a sum
  of either $\eta_{u\subL}$ or $\eta_{u\subR}$ for all $u\prec v$, and
  these are all independent. Furthermore, there are typically about as many
  left and right steps in the path to a node $v$, and it is easy to see that
  the distribution of $X_V$ for a uniformly random node $V$, after
  normalizing as above, has  the normal limit in \eqref{tbst1} and
  \eqref{tyuleb}, \cf{} \refR{Rannealed}.
\end{remark}

\begin{example}\label{Ebstrrt}
  Let $\etar=0$. Then each right child can be identified with its mother;
this reduces the binary Yule tree 
to the Yule tree, and the \bst{} to the
\rrt{},
and 
the theorems above (for external nodes)
reduce to Theorems \ref{Trrt}--\ref{TYule}.
(This reduction of the \bst{} to the \rrt{} identifies 
the nodes in the \rrt{} with
the external nodes in the \bst{}. 
It is related to the well-known rotation correspondence
\cite[p.~72]{Drmota},
which, however, 
identifies the nodes in the \rrt{} with the internal nodes in the \bst.)
\end{example}

\begin{proof}[Proof of Theorems \ref{Tbst} and \ref{TYuleb}]
  The arguments in Sections \ref{SpfTYule}--\ref{SpfTrrt} require only minor
  modifications. We therefore omit many details.
We assume again $d=1$.
We denote the \chf{} of $\etax$ by 
\begin{equation}
\gfx(s_1,s_2):=\E e^{\ii(s_1\eta\subL+s_2\eta\subR)},
\qquad s_1,s_2\in\bbR^d, 
\end{equation}
and the \chf{s} of the marginal distributions $\eta\subL$ and $\eta\subR$ by
\begin{align}
  \gf\subL(s)&:=\E e^{\ii s \eta\subL}=\gfx(s,0),
&
  \gf\subR(s)&:=\E e^{\ii s \eta\subR}=\gfx(0,s).
\end{align}

We consider  the binary Yule tree $\ctt$ and define, in addition to $F_t(s)$,
\begin{align}
  \label{FF}
\Fi_t(s)&:=\sum_{v\in \ctti} e^{\ii s X_v}, 
&
\Fe_t(s)&:=\sum_{v\in \ctte} e^{\ii s X_v}.
\end{align}
For the exterior version $\Fe_t$ we can argue as in \refS{SpfTYule}. 
We assume again first $d=1$.
Each
external node $v$ dies with intensity 1, and then gets two children. This
changes $\Fe_t(s)$ by
\begin{equation}\label{ea}
  -e^{\ii s X_v} + e^{\ii s X_\vl}+e^{\ii s X_\vr}
=
  e^{\ii s X_v}\bigpar{-1 + e^{\ii s \eta_\vl}+e^{\ii s \eta_\vr}}
\end{equation}
 For convenience, define
\begin{equation}\label{tgf}
\tgf(s):=\E \bigpar{e^{\ii s \eta_\vl}+e^{\ii s \eta_\vr}-1}
=\gfl(s)+\gfr(s)-1  
\end{equation}
and 
\begin{equation}
  \begin{split}
  \psi(s_1,s_2)&:=\E\bigsqpar{ \bigpar{e^{\ii s_1 \eta_\vl}+e^{\ii s_1 \eta_\vr}-1}
\bigpar{e^{\ii s_2 \eta_\vl}+e^{\ii s_2 \eta_\vr}-1}} 
\\&\phantom:
\begin{aligned}
  =\gfx(s_1,s_2)+\gfx(s_2,s_1)+\gfl(s_1+s_2)+\gfr(s_1+s_2)
\\\qquad-\gfl(s_1)-\gfl(s_2)-\gfr(s_1)-\gfr(s_2)+1
\end{aligned}
\\&\phantom:
 =\gfx(s_1,s_2)+\gfx(s_2,s_1)+\tgf(s_1+s_2)-\tgf(s_1)-\tgf(s_2).
  \end{split}
\end{equation}
It follows from \eqref{ea} that, \cf{} \eqref{er66},
\begin{equation}
  \label{er66b}
  \ddt \E \Fe_t(s) =  \bigpar{\gfl(s)+\gfr(s)-1} \E \Fe_t(s)
= \tgf(s)\E \Fe_t(s)
\end{equation}
and thus
\begin{equation}
  \label{EFe}
\E \Fe_t(s) =  e^{t\tgf(s) }.
\end{equation}
Similarly, 
\cf{} \eqref{erII} and \eqref{EFF},
\begin{equation}\label{erIIb}
  \begin{split}
&\ddt \E \bigpar{\Fe_{t}(s_1)\Fe_t(s_2)}
\notag\\&\qquad
=\bigpar{\tgf(s_1)+\tgf(s_2)}\E\bigpar{\Fe_t(s_1)\Fe_t(s_2)}
+ \psi(s_1,s_2)\E \Fe_t(s_1+s_2)
  \end{split}
\end{equation}
and hence, at least for $s_1,s_2$ in a suitably small interval $J=(-\gd,\gd)$,
    \begin{equation}\label{EFFb}
      \begin{split}
      \E \bigpar{\Fe_t(s_1)\Fe_t(s_2)} 
&= 
\frac{\gfx(s_1,s_2)+\gfx(s_2,s_1)}{\tgf(s_1)+\tgf(s_2)-\tgf(s_1+s_2)}
e^{t(\tgf(s_1)+\tgf(s_2))}
\\&\qquad{}-
\frac{\psi(s_1,s_2)}{\tgf(s_1)+\tgf(s_2)-\tgf(s_1+s_2)}
e^{t\tgf(s_1+s_2)}
.
      \end{split}
\end{equation}

The rest of the proof for $\bmue$ is as before.
It follows that $M_t(s):=\Fe(s)/\E \Fe(s)$ is an $L^2$-bounded
martingale in $\WWJ$ (if $J$ is small enough) and we obtain, \cf{}
\eqref{flex},
\begin{equation}\label{flexb}
  \hbmuet(s/\tqq)
=\frac{\Fe_t(s/\tqq)}{|\ctte|}
=\frac{\Fe_t(s/\tqq)}{\Fe_t(0)}
=e^{t(\tgf(s/\tqq)-1)+o(1)}.
\end{equation}
Furthermore, \eqref{tgf} implies
    \begin{equation}
      \begin{split}
\tgf(s)=1+\ii s (\E\etal +\E\etar)  
-\tfrac{s^2}2 \bigpar{\E [\etal\etal\tr]+ \E [\etar\etar\tr] }+o(s^2).
      \end{split}
    \end{equation}
Consequently,
\cf{} \eqref{kais},
it follows from \eqref{flexb} that,
again with $a:=\sqrt t$ and $b:=mt$
and
for any fixed $s\in\bbR$,
a.s.,
\begin{equation}\label{kaisbb}
  \widehat{\Thetaab (\bmue_t)}(s)
=e^{t(\tgf(s/\tqq)-1-\ii  \bm s/\tqq)+o(1)}
\to 
e^{-\frac{s^2}2 \xpar{\E [\etal\etal\tr]+ \E [\etar\etar\tr] }}.
\end{equation}
The proofs of \refT{TYuleb} and \ref{Tbst} for $\bmue$ are completed as
before.

To transfer the results to the internal version, we note that each internal
nod has two children, and that these children comprise all internal and
external nodes except the root.
Hence, 
\begin{equation}
  F_t(s)
=1+\sum_{v\in\ctti} \bigpar{e^{\ii s X_\vl}+e^{\ii s X_\vr}}
\end{equation}
and
\begin{align}
      \Fe_t(s)-\Fi_t(s)
&=  F_t(s)-2\Fi_t(s)
=1+\sum_{v\in\ctti} \bigpar{e^{\ii s X_\vl}+e^{\ii s X_\vr}-2e^{\ii s X_v}}
\nonumber\\&
=1+\sum_{v\in\ctti}e^{\ii s X_v} \bigpar{e^{\ii s \eta_\vl}+e^{\ii s \eta_\vr}
  -2}.
\end{align}
Thus,
\begin{equation}\label{borg}
  \begin{split}
|\Fe_t(s)-\Fi_t(s)| 
&\le
1+\sum_{v\in\ctti} \bigabs{e^{\ii s \eta_\vl}+e^{\ii s \eta_\vr}-2}
\\&
\le1+\sum_{v\in\ctti} \bigpar{|e^{\ii s \eta_\vl}-1|+|e^{\ii s \eta_\vr}-1|}
\\&
\le1+2\sum_{v\in\ctti} \bigpar{|s\eta_\vl|\land1+|s\eta_\vr|\land1}.
  \end{split}
\end{equation}
Conditioned on the tree process $(\ctt)_t$, the internal nodes are added one
by one, and thus the standard law of large numbers implies that for any
fixed $s\in\bbR$, as \ttoo,
\begin{equation}\label{bamse}
\frac{1}{|\ctti|}\sum_{v\in\ctti} \bigpar{|s\eta_\vl|\land1}
\asto \E [|s\etal|\land1].
\end{equation}
Let again $s(t):=s/\sqrt t$ for some fixed $s\in\bbR$. 
Then for any given $\eps>0$,
$|s(t)|\le\eps$ for all large $t$, and thus \eqref{bamse} implies that a.s.,
\begin{equation}\label{brumma}
\limsup_{\ttoo}\frac{1}{|\ctti|}\sum_{v\in\ctti} \bigpar{|s(t)\eta_\vl|\land1}
\le
\limsup_{\ttoo}\frac{1}{|\ctti|}\sum_{v\in\ctti}\bigpar{ |\eps\eta_\vl|\land1}
= \E [|\eps\etal|\land1].
\end{equation}
As $\eps\to0$, the \rhs{} of \eqref{brumma} tends to 0 by dominated
convergence, and thus \eqref{brumma} implies that
\begin{equation}\label{brummb}
\frac{1}{|\ctti|}\sum_{v\in\ctti} \bigpar{|s(t)\eta_\vl|\land1} \asto0.
\end{equation}
The same holds for $\eta_\vr$, and thus \eqref{borg} implies
\begin{equation}\label{borgg}
  \begin{split}
\biggabs{\frac{\Fe_t(s(t))}{|\ctti|}-\frac{\Fi_t(s(t))}{|\ctti|}} 
&
\le\frac{1}{|\ctti|}
+\frac{2}{|\ctti|}
\sum_{v\in\ctti}\bigpar{|s(t)\eta_\vl|\land1+|s(t)\eta_\vr|\land1}
\asto0.
  \end{split}
\end{equation}
Since $|\ctte|/|\ctti|=1+1/|\ctti|\to1$ a.s., 
\eqref{borgg} and \eqref{flexb} yield, a.s.,
\begin{equation}
  \begin{split}
  \hbmuit(s/\tqq)
=\frac{\Fi_t(s/\tqq)}{|\ctti|}
=\frac{\Fe_t(s/\tqq)}{|\ctte|}\frac{|\ctte|}{|\ctti|}+o(1)
=e^{t(\tgf(s/\tqq)-1)}+o(1).
  \end{split}
\end{equation}
It follows that \eqref{kaisbb} holds also for $\bmui$, and
the results for $\bmui$ now follow
by the arguments used above for $\bmue$
(and earlier in Sections \ref{SpfTYule}--\ref{SpfTrrt}).

Finally, 
\begin{equation}
  \bmu_t = \frac{|\ctti|}{|\ctt|}\bmui_t+\frac{|\ctte|}{|\ctt|}\bmue_t,
\end{equation}
and the results for $\bmu$ follow from the results for $\bmui$ and $\bmue$.
\end{proof}

\begin{example}\label{Eprof}
  Let $\etax=(-1,1)$ (deterministically). Then $X_v$ is the difference
  between the number of right and left steps in the path from the root to
  $v$,
and $\bmu$ (or $\bmui$) is known as the vertical profile of $T$.
\refT{Tbst} shows that the vertical profile \as{} is asymptotically normal,
with
\begin{equation}\label{prof}
  \Theta_{\sqrt{\log n},\,0}(\bmu_n) \asto
N\xpar{0,2}
,
\end{equation}
as proved by \citet{KubaPanholzer}.
This can be compared to (completely different) results for
uniformly random binary trees by
\citet{Marckert} and \citet{SJ185}.
\end{example}

\section{Some open problems}\label{Sproblem}

The results above suggest some possible directions for future research.

\subsection{Stable limits?}

We assume throughout the present paper 
that the offsets have finite second moments.
If this does not hold, we may still conjecture that 
asymptotically,
the empirical distribution $\bmu$ \as{} is as the distribution of a sum of
independent copies of $\eta$, cf.\ \refR{Rpoisson}. 
In particular, this leads to the following problem.
\begin{problem}
  Suppose that the offsets have a distribution in the domain of attraction
  of a stable  distribution $\gL$ with index $\ga<2$. Does 
$\bmu$ after suitable normalization converge \as{} to $\gL$?
\end{problem}

\citet{Fekete} has shown such results 
with convergence in probability
for the \bst.

In the proof above, the second moment is used not only in the final part (\cf{}
\refR{Rpoisson}), but also in showing that $M_t\in\WWJ$ and in the proof of
\refL{LMWWJ}. Hence some new idea is needed for this problem.

Note that it is easy to show that the annealed variable $X_{V_n}$ has the
asymptotic distribution $\gL$, \cf{} Remarks \ref{Rannealed} and \ref{Rbins}.

\subsection{Fluctuations?}

The theorems above yield convergence to a deterministic limit.
What can be said about the fluctuations?
For example:

\begin{problem}
Let $d=1$ and let $\gsx^2:=\E\eta^2\in(0,\infty)$.
For $x\in\bbR$ and $n\ge1$, let $I_{n,x}:=(-\infty, m\log n + x\sqrt{\log n}]$.
Then \refT{Trrt} says that \as, for every fixed $x\in\bbR$,
\begin{equation}\label{delta}
  \bmu_n(I_{n,x}) \to \Phi\bigpar{x/\gsx},
  \end{equation}
where $\Phi$ is the standard normal distribution function.
What can be said about the distribution of the difference between the two
sides in \eqref{delta}?
Is it asymptotically normal?
\end{problem}



\newcommand\AAP{\emph{Adv. Appl. Probab.} }
\newcommand\JAP{\emph{J. Appl. Probab.} }
\newcommand\JAMS{\emph{J. \AMS} }
\newcommand\MAMS{\emph{Memoirs \AMS} }
\newcommand\PAMS{\emph{Proc. \AMS} }
\newcommand\TAMS{\emph{Trans. \AMS} }
\newcommand\AnnMS{\emph{Ann. Math. Statist.} }
\newcommand\AnnPr{\emph{Ann. Probab.} }
\newcommand\CPC{\emph{Combin. Probab. Comput.} }
\newcommand\JMAA{\emph{J. Math. Anal. Appl.} }
\newcommand\RSA{\emph{Random Struct. Alg.} }
\newcommand\ZW{\emph{Z. Wahrsch. Verw. Gebiete} }
\newcommand\DMTCS{\jour{Discr. Math. Theor. Comput. Sci.} }

\newcommand\AMS{Amer. Math. Soc.}
\newcommand\Springer{Springer-Verlag}
\newcommand\Wiley{Wiley}

\newcommand\vol{\textbf}
\newcommand\jour{\emph}
\newcommand\book{\emph}
\newcommand\inbook{\emph}
\def\no#1#2,{\unskip#2, no. #1,} 
\newcommand\toappear{\unskip, to appear}

\newcommand\arxiv[1]{\texttt{arXiv:#1}}
\newcommand\arXiv{\arxiv}

\def\nobibitem#1\par{}

\end{document}